\documentclass[final]{elsarticle}

\usepackage[margin=1.5in]{geometry}
\usepackage{amssymb}
\usepackage{amsthm}
\usepackage{latexsym}
\usepackage{amsfonts}
\usepackage{amsmath}
\usepackage[latin1]{inputenc}
\usepackage{subfigure}
\usepackage{setspace}
\usepackage{graphicx}
\usepackage{graphics}
\usepackage{multirow}
\usepackage{color}
\usepackage{epsfig}
\usepackage{soul} 
\usepackage{mdframed}

\usepackage{amsrefs}
\usepackage{hyperref}

\newtheorem{lemma}{Lemma}

\newtheorem{example}{Example}
\newtheorem{remark}{Remark}

\newcommand{\schmal} {4.2}
\newcommand{\breite} {12.95}

\usepackage{tikz}
\usetikzlibrary{positioning,shapes}

\definecolor{arrowblue}{RGB}{98,145,224}

\journal{arXiv}

\begin{document}

\begin{frontmatter}



\title{A second-order exponential time differencing scheme for non-linear reaction-diffusion systems with dimensional splitting}


\author[1]{E.O. Asante-Asamani}
\ead{nasmano@gmail.com}
\address[1]{Hunter College, City University of New York, New York City, NY 10065, USA}

\author[2]{A. Kleefeld\corref{cor1}}
\ead{a.kleefeld@fz-juelich.de}
\address[2]{J\"ulich Supercomputing Centre, Forschungszentrum J\"ulich GmbH, 52425 J\"ulich, Germany}
\cortext[cor1]{corresponding author}

\author[3]{B.A. Wade}
\ead{bruce.wade@louisiana.edu}
\address[3]{Department of Mathematics, University of Louisiana at Lafayette, Lafayette, LA 70504-3568, USA}

\begin{abstract}
A second-order $L$-stable exponential time-differencing (ETD) method is developed by combining an ETD scheme with approximating the matrix exponentials by rational functions 
having real distinct poles (RDP), together with a dimensional splitting integrating factor technique. A variety of non-linear reaction-diffusion equations in two and three dimensions with either Dirichlet, Neumann, or 
periodic boundary conditions are solved with this scheme and shown to outperform a variety of other second-order implicit-explicit schemes. An additional performance boost is gained 
through further use of basic parallelization techniques.
 \end{abstract}


\end{frontmatter}


\section{Introduction}
Many applications, for example in biology \cites{Murray2002,Murray2003,Mueller2015}, pattern formation \cite{Kondo2010}, or medicine \cites{Holmes21994,Sherrat1990} 
can be modelled with a system of $s$ time-dependent non-linear reaction-diffusion equations in $d$ dimensions given by
\begin{eqnarray}
\frac{\partial \mathbf{u}}{\partial t} = \mathbf{D}\Delta \mathbf{u} + \mathbf{F}(\mathbf{u})\,,\qquad \mathbf{x}\in\Omega\subset \mathbb{R}^d\,,\quad t\in (0,T)
\end{eqnarray}
with appropriate initial and boundary condition for the domain $\Omega$ and the time interval $(0,T)$. The numerical computation of such systems in two or three dimensions is challenging due to 
its high dimensionality, coupling, stiffness of the reaction and diffusion terms, and since the function $\mathbf{F}$ may be non-linear.
Additionally, spurious oscillations in the approximate solution might occur if the initial data are non-smooth or are incompatible with the boundary data.

Many algorithms have been proposed to solve such systems after discretization in space. 
For example, linearly implicit methods \cites{Akrivis1999,Vigo2007,Wade2001}, semi-implicit methods \cite{Chen1998}, projection methods \cite{Gear2003},  
stabilized explicit Runge-Kutta methods such as ROCK2 and ROCK4 \cites{rock2,rock4}, PIROCK \cite{abdullepirock}, SERK, SERK2V2, and SERK2V3 \cites{serk2,serkv2,multi,serk2v3}, its extrapolated variants \cites{eserk4,eserk5,eserk6}, 
explicit Runge-Kutta-Chebyshev (RKC) \cite{rkc}, implicit-explicit Runge-Kutta-Chebyshev (IRKC) \cite{irkc} (one of many implicit-explicit methods (IMEX) \cites{Hundsdorfer,ruuth}), 
factorized Runge-Kutta-Chebyshev (FRKC) or 
Runge-Kutta with Gegenbauer polynomials \cites{Sul15,Sul19}, integration factor methods \cites{wang,Zhao} in combination with Krylov subspace methods \cites{chen2011,donglu}, 
and exponential propagation iterative methods of Runge-Kutta type (EPIRK) \cite{tokman} 
and its variants such as IMEXP \cite{luan}, to mention a few.

One important family to solve the problem at hand is exponential time differencing (ETD), which uses the Duhamel principle and solves the resulting initial value problem by approximating the integral 
appropriately \cites{Wade2009,Wade2012}.
A class of ETD schemes based on Runge-Kutta methods has been introduced by Cox and Matthews \cite{Cox2002}. Since then, many researchers have considered this scheme 
(see for example \cites{Du2005,Hoch12005,hoch2005,minchev,Wade2009,kassamtrefethen}). However, the challenge of how to effectively approximate the matrix exponentials remains.
Kleefeld, Khaliq, and Wade \cites{Kleefeld2012,etdoption} considered an ETD Crank-Nicolson (ETD-CN) scheme; that is, the matrix exponential has been approximated through a second-order Pad\'e-$(1,1)$ approximation. 
Their fully discrete scheme is of second-order and highly 
efficient, but it is not $L$-stable. Other schemes use Krylov subspace methods to approximate the matrix exponential \cite{bhatt2018} or use higher order Pad\'e approximations \cite{vigo2007a}.

To speed up computations and to decrease memory demand, various dimensional splitting techniques such as, Strang simple and Strang symmetric splitting, as well as an 
integrating factor method have been considered by Asante-Asamani and Wade \cite{asante2} for a variety of examples with Dirichlet and Neumann boundary conditions in two dimensions
as a competitive alternative to the locally one-dimensional splitting of Bhatt and Khaliq \cite{bhatt}. It has been observed empirically that the integrating factor approach outperforms other splitting techniques.

An extension of the work for ETD-CN with Pad\'e-$(1,1)$ is given by Yousuf, Khaliq, and Kleefeld \cite{Yousuf2012} using a Pad\'e-$(0,2)$ approximation of the matrix exponential leading to a second-order $L$-stable scheme.
However, one has to use complex arithmetic in the implementation due to the two complex-valued poles in the Pad\'e approximation. 
To avoid this, one can use a rational (non-Pad\'e) approximation of second-order with real distinct poles (see \cite{Khaliq1994}), which has recently been applied by Asante-Asamani, Khaliq and Wade \cite{asante} to 
the second-order ETD scheme, named hereafter ETD-RDP. It is shown that ETD-RDP is a competitive alternative to BDF2 \cite{Vigo2007} and ETD-Pad\'e-$(0,2)$ for a variety of
examples with Dirichlet and Neumann boundary conditions in two dimensions. 
Further, it has been remarked by the authors that the ETD-RDP algorithm is easily parallelizable.

\section*{Contribution}
In this work, we apply a splitting technique using an integrating factor (IF) approach to the second-order $L$-stable ETD-RDP scheme, lifting the derivation from two dimensions to $d$ dimensions, which we 
call ETD-RDP-IF. We show that besides Dirichlet and Neumann boundary conditions, also the periodic boundary condition case can be handled. 
Further, the algorithm can deal with non-smooth boundary conditions as well as with mismatching boundary conditions.
Additionally, we 
explain how to solve the sparse linear systems in higher dimensions efficiently using an extension of the Thomas algorithm for Dirichlet and Neumann boundary conditions and 
using the Fourier transformation for periodic boundary conditions. Further, we describe how to implement a parallelization strategy in Fortran, and show that we developed a second-order $L$-stable method that is able to outperform other second-order IMEX schemes as well as ETD-RDP.

\section*{Organization of the paper}
In Section \ref{section2}, the mathematical model of a system consisting of $s$ non-linear reaction-diffusion equations is detailed.
Section \ref{section3} illustrates how this system is discretized in $d$-dimensional space. The dimensional splitting using the integrating factor approach is given in 
Section \ref{section4}. The discretization with respect to time using exponential time differencing is described in Section \ref{section5}. In Section \ref{section6}, we develop the approximation of the matrix exponential. 
The implementation and parallelization of the fully discrete ETD-RDP-IF scheme is given in Section \ref{section7}. Extensive numerical results in two and three dimensions with Dirichlet, Neumann, and periodic boundary conditions 
are given in Section \ref{section8}, which show the second-order accuracy and the efficiency of the new ETD-RDP-IF scheme. A short summary and a conclusion together with an outlook for possible future research is given in 
Section \ref{section9}.

\section{Time-dependent non-linear reaction-diffusion equation}\label{section2}
The mathematical model of a system of $s$ time-dependent non-linear reaction-diffusion equations in $d$ dimensions is given as
\begin{eqnarray}
\frac{\partial \mathbf{u}}{\partial t} = \mathbf{D}\Delta \mathbf{u} + \mathbf{F}(\mathbf{u})\,,\qquad \mathbf{x}\in\Omega\subset \mathbb{R}^d\,,\quad t\in (0,T)
\label{model}
\end{eqnarray}
where the vector-valued concentration function $\mathbf{u}=\mathbf{u}(\mathbf{x},t)=(u_1(\mathbf{x},t),\ldots,u_s(\mathbf{x},t))^\top$ depends on the spatial variable $\mathbf{x}=(x_1,\ldots,x_d)\in \Omega$ 
and the temporal variable $t\in (0,T)$. Here,
$\Omega\subset\mathbb{R}^d$ is a bounded domain with Lipschitz continuous boundary. Each component of the vector-valued function $\mathbf{F}(\mathbf{u})=(f_1(\mathbf{x}),\ldots,f_s(\mathbf{x}))^\top$ is assumed to
be a sufficiently smooth and bounded function which may be non-linear. The diagonal matrix $\mathbf{D}=\mathrm{diag}(D_1,\ldots,D_s)\in \mathbb{R}^{s\times s}$ contains 
given constant diffusion coefficients and $\Delta \mathbf{u}$ is the $d$-dimensional Laplacian taken component-wise. The boundary conditions on $\partial \Omega$ are either homogeneous Dirichlet $\mathbf{u}(\cdotp,t)=0$, 
homogeneous Neumann $\frac{\partial}{\partial \nu}\mathbf{u}(\cdotp,t)=0$ where $\nu$ denotes the exterior normal of $\Omega$, or periodic boundary conditions. The initial condition is given by
$\mathbf{u}(\mathbf{x},0)=\mathbf{u}_0(\mathbf{x})$ for $\mathbf{x}\in \Omega$.
\section{Discretizing in space}\label{section3}
As a first step, we discretize the spatial domain assuming an equidistant grid size $h$ along all $d$ directions. The parameter $p$ denotes the number of spatial grid points along each direction. 
Then, the second-order derivatives within the expression $-\mathbf{D}\Delta \mathbf{u}$ are discretized by centered differences with 
second-order accuracy. As a result, one obtains for (\ref{model}) the system of non-linear ordinary differential equations
\begin{eqnarray}
 \frac{\partial \mathbf{U}}{\partial t} +\mathbf{A}\mathbf{U} = \mathbf{f}(\mathbf{U})\,,\quad \mathbf{U}(0)=\mathbf{U}_0
 \label{start2}
\end{eqnarray}
where $\mathbf{U}$ is a vector of size $m=s\cdotp p^d$, $\mathbf{A}$ is a matrix of size $s\cdotp p^d$ times $s\cdotp p^d$, $\mathbf{f}(\mathbf{U})$ is a 
vector-valued function of size $s\cdotp p^d$ ($\mathbf{F}(\mathbf{u})$ evaluated at the $p^d$ spatial points), and the initial condition 
$\mathbf{U}_0$ is a vector of size $s\cdotp p^d$ ($\mathbf{u}_0(\mathbf{x})$ evaluated at the $p^d$ spatial points). 
The entries of the matrix $\mathbf{A}$ depend on the particular boundary condition. Three cases are given in the following two examples.
\begin{example}\label{ex1} The one-dimensional approximation of the $-\Delta \mathbf{u}$ is given by
\begin{eqnarray*}
\mathbf{B}_p &=&-\frac{1}{h^2}\left(\begin{array}{rrrrr}
-2 & 1 &  & &\\
1 & -2 & 1 & &\\
  & \ddots & \ddots & \ddots &\\
   && 1 & -2 & 1\\
  &&  & 1 & -2 \end{array}\right)  \in \mathbb{R}^{p\times p}\,, \qquad  h = \frac{1}{p+1}\,,\\
\mathbf{B}_p &=&- \frac{1}{h^2}\left(\begin{array}{rrrrr}
-2 & 2 &  & &\\
1 & -2 & 1 & &\\
  & \ddots & \ddots & \ddots&\\
  && 1 & -2 & 1\\
  &&  & 2 & -2 \end{array}\right)  \in \mathbb{R}^{p\times p}\,, \qquad h = \frac{1}{p-1}\,,\\
\mathbf{B}_p &=&- \frac{1}{h^2}\left(\begin{array}{rrrrr}
-2 & 1 &   & &1\\
1 & -2 & 1 & &\\
 &\ddots & \ddots & \ddots&\\
  && 1 & -2 & 1\\
 1& &   & 1 & -2  \end{array}\right) \in \mathbb{R}^{p\times p}\,, \qquad h = \frac{1}{p}\,,
\end{eqnarray*}
for homogeneous Dirichlet, homogeneous Neumann, and periodic boundary conditions, respectively. Hence, in one dimension $\mathbf{A}$ in (\ref{start2}) approximating $-\mathbf{D}\Delta \mathbf{u}$ is given by $\mathbf{A}=\mathbf{A}_1$ with 
$\mathbf{A}_1=\mathbf{B}_p\otimes \mathbf{D}\in\mathbb{R}^{(s\cdotp p)\times (s\cdotp p)}$. 
Here, $\otimes$ denotes the Kronecker product (see \cite{loan} for the definition and its properties).
\end{example}

\begin{example}\label{ex2}
In two dimensions, $\mathbf{A}$ in (\ref{start2}) is given by $\mathbf{A}=\mathbf{A}_1+\mathbf{A}_2$ with 
\begin{eqnarray*}
 \mathbf{A}_1&=&\mathbf{I}_p\otimes \mathbf{B}_p\otimes \mathbf{D}\in\mathbb{R}^{(s\cdotp p^2)\times (s\cdotp p^2)}\,,\\
 \mathbf{A}_2&=&\mathbf{B}_p\otimes \mathbf{I}_p\otimes \mathbf{D}\in\mathbb{R}^{(s\cdotp p^2)\times (s\cdotp p^2)}\,,
\end{eqnarray*}
where $\mathbf{I}_p$ denotes the identity matrix of size $p$
(see also pp. 1345--1346 in \cite{asante2} for the Dirichlet and Neumann case). 
The extension to three dimensions reads $\mathbf{A}=\mathbf{A}_1+\mathbf{A}_2+\mathbf{A}_3$ with 
\begin{eqnarray}
 \mathbf{A}_1&=&\mathbf{I}_p\otimes\mathbf{I}_p\otimes \mathbf{B}_p\otimes \mathbf{D}\in\mathbb{R}^{(s\cdotp p^3)\times (s\cdotp p^3)}\,,\nonumber\\
 \mathbf{A}_2&=&\mathbf{I}_p\otimes\mathbf{B}_p\otimes \mathbf{I}_p\otimes \mathbf{D}\in\mathbb{R}^{(s\cdotp p^3)\times (s\cdotp p^3)}\,,\nonumber\\
 \mathbf{A}_3&=&\mathbf{B}_p\otimes \mathbf{I}_p\otimes \mathbf{I}_p\otimes\mathbf{D}\in\mathbb{R}^{(s\cdotp p^3)\times (s\cdotp p^3)}\,.
 \label{3dcase}
\end{eqnarray}
Note that the extension to four dimensions or even 
higher follows an obvious pattern. Precisely, we have $\mathbf{A}=\sum_{i=1}^d\mathbf{A}_i$ with 
\begin{eqnarray}
\mathbf{A}_i=\underbrace{\mathbf{I}_p\otimes \cdots\otimes \mathbf{I}_p}_{(d-i) \text{times}}\otimes\underbrace{\mathbf{B}_p}_{\text{pos } d-i+1} 
\otimes\underbrace{\mathbf{I}_p \otimes  \cdots \otimes \mathbf{I}_p}_{(i-1) \text{times}}\otimes\mathbf{D}\in\mathbb{R}^{(s\cdotp p^d)\times (s\cdotp p^d)}\,.
\label{Ai}
\end{eqnarray}
\end{example}
\section{Dimensional splitting}\label{section4}
 In this section, we explain how to apply a dimensional splitting technique to solve (\ref{start2}). Precisely, we here use $\mathbf{A}=\mathbf{A}_1+\mathbf{A}_2+\mathbf{A}_3$.
 Consider a new time-dependent function of the form $\mathbf{V}=\mathrm{e}^{\mathbf{A}_1 t}\mathbf{U}$ (here the term $\mathrm{e}^{\mathbf{A}_1 t}$ is called an integrating factor).
 The matrix exponential is described in \cite{moler}. Then, 
 the derivative with respect to time $t$ of 
 $\mathbf{V}=\mathrm{e}^{\mathbf{A}_1 t}\mathbf{U}$ is given by 
 \begin{eqnarray}
 \mathbf{V}_t=\mathrm{e}^{\mathbf{A}_1 t}\mathbf{U}_t+\mathbf{A}_1\mathrm{e}^{\mathbf{A}_1 t}\mathbf{U}
 \label{step1}
 \end{eqnarray}
 recalling that $\mathbf{U}$ depends on $t$ as well. Inserting (\ref{start2}) into (\ref{step1}) yields
 \begin{eqnarray*}
  \mathbf{V}_t&=&\mathrm{e}^{\mathbf{A}_1 t}\left(\mathbf{f}(\mathbf{U})- \mathbf{A}\mathbf{U}\right)+\mathbf{A}_1\mathrm{e}^{\mathbf{A}_1 t}\mathbf{U}\\
  &=& \mathrm{e}^{\mathbf{A}_1 t}\mathbf{f}(\mathbf{U})-\mathrm{e}^{\mathbf{A}_1 t}\mathbf{A}\mathbf{U}+\mathbf{A}_1\mathrm{e}^{\mathbf{A}_1 t}\mathbf{U}\\
  &=& \mathrm{e}^{\mathbf{A}_1 t}\mathbf{f}(\mathbf{U})-\mathbf{A}\mathrm{e}^{\mathbf{A}_1 t}\mathbf{U}+\mathbf{A}_1\mathrm{e}^{\mathbf{A}_1 t}\mathbf{U}\\
  &=& \mathrm{e}^{\mathbf{A}_1 t}\mathbf{f}(\mathbf{U})-\left(\mathbf{A}_2+\mathbf{A}_3\right)\mathrm{e}^{\mathbf{A}_1 t}\mathbf{U}\\
  &=& \mathrm{e}^{\mathbf{A}_1 t}\mathbf{f}(\mathrm{e}^{-\mathbf{A}_1 t}\mathbf{V})-\left(\mathbf{A}_2+\mathbf{A}_3\right)\mathbf{V}
 \end{eqnarray*}
 where we used in the third step that $\mathbf{A}$ and $\mathbf{A}_1$ commute, 
 hence it follows that $\mathbf{A}$ and $\mathrm{e}^{\mathbf{A}_1 t}$ commute. This is proved below. Hence, we obtain
 \begin{eqnarray}
 \mathbf{V}_t+\left(\mathbf{A}_2+\mathbf{A}_3\right)\mathbf{V}=\mathbf{g}(\mathbf{V})\,,\quad \mathbf{V}(0)=\mathbf{U}_0
 \label{new}
 \end{eqnarray}
 with $\mathbf{g}(\mathbf{V})=\mathrm{e}^{\mathbf{A}_1 t}\mathbf{f}(\mathrm{e}^{-\mathbf{A}_1 t}\mathbf{V})$ and the initial condition $\mathbf{V}(0)=\mathrm{e}^{\mathbf{A}_1 0}\mathbf{U}_0=\mathbf{I}_m\mathbf{U}_0=\mathbf{U}_0$ 
 with $m=s\cdotp p^3$. Similarly, we now define $\mathbf{W}=\mathrm{e}^{\mathbf{A}_2 t}\mathbf{V}$, compute $\mathbf{W}_t$, insert $\mathbf{V}_t$ from (\ref{new}) into $\mathbf{W}_t$, and use the fact that $\mathrm{e}^{\mathbf{A}_2}$ 
 commutes with $\mathbf{A}_2+\mathbf{A}_3$ 
 which is true since $\mathbf{A}_2$ and $\mathbf{A}_2+\mathbf{A}_3$ commute as shown below. We obtain
 \begin{eqnarray}
 \mathbf{W}_t+\mathbf{A}_3\mathbf{W}=\mathbf{h}(\mathbf{W})\,,\quad \mathbf{W}(0)=\mathbf{U}_0\,,
 \label{new2}
 \end{eqnarray}
 where $\mathbf{h}(\mathbf{W})=\mathrm{e}^{\mathbf{A}_2 t}\mathbf{g}(\mathrm{e}^{-\mathbf{A}_2 t}\mathbf{W})$.
Once $\mathbf{W}(t)$ is found, we can obtain $\mathbf{V}(t)$ using the formula $\mathbf{V}(t)=\mathrm{e}^{-\mathbf{A}_2 t}\mathbf{W}(t)$ and $\mathbf{U}(t)$ via 
$\mathbf{U}(t)=\mathrm{e}^{-\mathbf{A}_1 t}\mathbf{V}(t)$. Hence, the original three-dimensional problem is converted to a problem involving only one-dimensional components.
\begin{remark}
The solution of (\ref{start2}) with $\mathbf{A}=\sum_{i=1}^d\mathbf{A}_i$ where $\mathbf{A}_i$ is given by (\ref{Ai}) can be obtained via splitting  through
 \begin{eqnarray*}
  \mathbf{Z}_t+\mathbf{A}_d\mathbf{Z}=\mathbf{k}(\mathbf{Z})\,,\quad \mathbf{Z}(0)=\mathbf{U}_0\,,
 \end{eqnarray*}
 where $\mathbf{k}(\mathbf{Z})=\left(\mathrm{e}^{\mathbf{A}_{d-1} t}\bullet\ldots\bullet\mathrm{e}^{\mathbf{A}_1 t}\right)\mathbf{F}( \left(\mathrm{e}^{-\mathbf{A}_{1} t}\bullet\ldots\bullet\mathrm{e}^{-\mathbf{A}_{d-1} t}\right)\mathbf{Z})$,
 where $\bullet$ denotes matrix multiplication (whenever needed for clarity). 
 Then, the function $\mathbf{U}(t)$ is given by the expression
 $\mathbf{U}(t)=\left(\mathrm{e}^{-\mathbf{A}_{1} t}\bullet\ldots\bullet\mathrm{e}^{-\mathbf{A}_{d-1} t}\right)\mathbf{Z}(t)$ where it is assumed that the operators commute accordingly. We will see later that this is the case for our examples.
\end{remark}
All that is left to show is that $\mathbf{A}$ and $\mathbf{A}_1$ commute as well as that $\mathbf{A}_2$ and $\mathbf{A}_2+\mathbf{A}_3$ commute for the three-dimensional case.
\begin{lemma}
Let $\mathbf{A}=\mathbf{A}_1+\mathbf{A}_2+\mathbf{A}_3$, where $\mathbf{A}_i$ is defined in (\ref{3dcase}). Then
 $\mathbf{A}$ and $\mathbf{A}_1$ commute. Also $\mathbf{A}_2$ and $\mathbf{A}_2+\mathbf{A}_3$ commute.
\end{lemma}

\begin{proof}
 It suffices to show that $\mathbf{A}_2\mathbf{A}_1=\mathbf{A}_1\mathbf{A}_2$, $\mathbf{A}_1\mathbf{A}_3=\mathbf{A}_3\mathbf{A}_1$, and $\mathbf{A}_2\mathbf{A}_3=\mathbf{A}_3\mathbf{A}_2$.
 Using the product property of the Kronecker product, we obtain
 \begin{align*}
\mathbf{A}_1\mathbf{A}_2 &= (\mathbf{I}_p \otimes \mathbf{I}_p\otimes \mathbf{B}_p \otimes \mathbf{D} )( \mathbf{I}_p \otimes \mathbf{B}_p\otimes \mathbf{I}_p \otimes \mathbf{D})\\
       &=\mathbf{I}_p\mathbf{I}_p \otimes \mathbf{I}_p\mathbf{B}_p\otimes \mathbf{B}_p\mathbf{I}_p \otimes \mathbf{D}^2 \\
       &= \mathbf{I}_p \otimes \mathbf{B}_p\otimes \mathbf{B}_p \otimes \mathbf{D}^2
 \end{align*}      
 and similarly
 \begin{align*}
\mathbf{A}_2\mathbf{A}_1 &= ( \mathbf{I}_p \otimes \mathbf{B}_p\otimes \mathbf{I}_p \otimes \mathbf{D})(\mathbf{I}_p \otimes \mathbf{I}_p\otimes \mathbf{B}_p \otimes \mathbf{D} )\\
       &= \mathbf{I}_p\mathbf{I}_p \otimes \mathbf{I}_p\mathbf{B}_p\otimes \mathbf{B}_p\mathbf{I}_p \otimes \mathbf{D}^2\\
       &= \mathbf{I}_p \otimes \mathbf{B}_p\otimes \mathbf{B}_p \otimes \mathbf{D}^2\,.
\end{align*}
We also have
\begin{align*}
\mathbf{A}_1\mathbf{A}_3 &= (\mathbf{I}_p \otimes \mathbf{I}_p\otimes \mathbf{B}_p \otimes \mathbf{D} )( \mathbf{B}_p \otimes \mathbf{I}_p\otimes \mathbf{I}_p \otimes \mathbf{D})\\
       &= \mathbf{B}_p \otimes \mathbf{I}_p\otimes \mathbf{B}_p \otimes \mathbf{D}^2 \\
\mathbf{A}_3\mathbf{A}_1 &= ( \mathbf{B}_p \otimes \mathbf{I}_p\otimes \mathbf{I}_p \otimes \mathbf{D})(\mathbf{I}_p \otimes \mathbf{I}_p\otimes \mathbf{B}_p \otimes \mathbf{D} )\\
       &= \mathbf{B}_p \otimes \mathbf{I}_p\otimes \mathbf{B}_p \otimes \mathbf{D}^2\\
\mathbf{A}_2\mathbf{A}_3 &= (\mathbf{I}_p \otimes \mathbf{B}_p\otimes \mathbf{I}_p \otimes \mathbf{D} )( \mathbf{B}_p \otimes \mathbf{I}_p\otimes \mathbf{I}_p \otimes \mathbf{D})\\
       &= \mathbf{B}_p \otimes \mathbf{B}_p\otimes \mathbf{I}_p \otimes \mathbf{D}^2\\
\mathbf{A}_3\mathbf{A}_2 &= ( \mathbf{B}_p \otimes \mathbf{I}_p\otimes \mathbf{I}_p \otimes \mathbf{D})(\mathbf{I}_p \otimes \mathbf{B}_p\otimes \mathbf{I}_p \otimes \mathbf{D} )\\
       &= \mathbf{B}_p \otimes \mathbf{B}_p\otimes \mathbf{I}_p \otimes \mathbf{D}^2\,.
       \end{align*}
\end{proof}
\begin{lemma}\label{lemma2}
 For any $i,j\in \{1,2\ldots,d\}$, the matrices $\mathbf{A}_i$ commute with $\mathbf{A}_j$, where they are given in (\ref{Ai}).
\end{lemma}
\begin{proof}
 We have on the one hand
 \begin{align*}
&\mathbf{A}_i\mathbf{A}_j\\
&=\left(\underbrace{\mathbf{I}_p\otimes \cdots\otimes \mathbf{I}_p}_{(d-i) \text{times}}\otimes\underbrace{\mathbf{B}_p}_{\text{pos } d-i+1}
\otimes\underbrace{\mathbf{I}_p \otimes  \cdots \otimes \mathbf{I}_p}_{(i-1) \text{times}}\otimes\mathbf{D}\right)\\
&\bullet\left(\underbrace{\mathbf{I}_p\otimes \cdots\otimes \mathbf{I}_p}_{(d-j) \text{times}}\otimes\underbrace{\mathbf{B}_p}_{\text{pos } d-j+1}
\otimes\underbrace{\mathbf{I}_p \otimes  \cdots \otimes \mathbf{I}_p}_{(j-1) \text{times}}\otimes\mathbf{D}\right)\\
&=\mathbf{I}_p \otimes  \cdots \otimes \mathbf{I}_p \otimes \underbrace{\mathbf{B}_p}_{\text{pos } d-\min(i,j)+1}\otimes \mathbf{I}_p \otimes  \cdots \otimes \mathbf{I}_p \otimes 
\underbrace{\mathbf{B}_p}_{\text{pos } d-\max(i,j)+1}\otimes \mathbf{I}_p \otimes  \cdots \otimes \mathbf{I}_p \otimes \mathbf{D}^2\\
&=\mathbf{I}_p \otimes  \cdots \otimes \mathbf{I}_p \otimes \underbrace{\mathbf{B}_p}_{\text{pos } d-\min(j,i)+1}\otimes \mathbf{I}_p \otimes  \cdots \otimes \mathbf{I}_p \otimes 
\underbrace{\mathbf{B}_p}_{\text{pos } d-\max(j,i)+1}\otimes \mathbf{I}_p \otimes  \cdots \otimes \mathbf{I}_p \otimes \mathbf{D}^2
 \end{align*}
 and on the other hand
 \begin{align*}
&\mathbf{A}_j\mathbf{A}_i\\
&=\left(\underbrace{\mathbf{I}_p\otimes \cdots\otimes \mathbf{I}_p}_{(d-j) \text{times}}\otimes\underbrace{\mathbf{B}_p}_{\text{pos } d-j+1}
\otimes\underbrace{\mathbf{I}_p \otimes  \cdots \otimes \mathbf{I}_p}_{(j-1) \text{times}}\otimes\mathbf{D}\right)\\
&\bullet\left(\underbrace{\mathbf{I}_p\otimes \cdots\otimes \mathbf{I}_p}_{(d-i) \text{times}}\otimes\underbrace{\mathbf{B}_p}_{\text{pos } d-i+1}
\otimes\underbrace{\mathbf{I}_p \otimes  \cdots \otimes \mathbf{I}_p}_{(i-1) \text{times}}\otimes\mathbf{D}\right)\\
&=\mathbf{I}_p \otimes  \cdots \otimes \mathbf{I}_p \otimes \underbrace{\mathbf{B}_p}_{\text{pos } d-\min(j,i)+1}\otimes \mathbf{I}_p \otimes  \cdots \otimes \mathbf{I}_p \otimes 
\underbrace{\mathbf{B}_p}_{\text{pos } d-\max(j,i)+1}\otimes \mathbf{I}_p \otimes  \cdots \otimes \mathbf{I}_p \otimes \mathbf{D}^2
 \end{align*}
 which finishes the proof.
\end{proof}
\section{Discretization in time}\label{section5}
 What remains is to discretize (\ref{new2}) with respect to time. We now focus on the three-dimensional case, but the results easily extend to higher dimensions. The case for 
 two or even one-dimension is then a special case of the results for the three-dimensional case. 
Consider solving (\ref{new2}) using Duhamel's principle (see Theorem 2.5.4 in \cite{Zheng2004}); that is
\begin{equation*}
\mathbf{W}(t)= \mathrm{e}^{-\mathbf{A}_3 t}\mathbf{W}(0) + \int_0^t \mathrm{e}^{-\mathbf{A}_3(t-s)}\mathbf{h}(\mathbf{W}(\tau))\;\mathrm{d}s
\end{equation*}
which satisfies the semi-discrete equation in the interval $[t_n,t_{n+1}]$ with $k=t_{n+1}-t_n$ the positive time step. Thus,
\begin{equation*}
\mathbf{W}(t_{n+1})= \mathrm{e}^{-k\mathbf{A}_3}\mathbf{W}(t_n) + \int_{t_n}^{t_{n+1}} \mathrm{e}^{-\mathbf{A}_3(t_{n+1}-s)} \mathbf{h}(\mathbf{W}(s))\;\mathrm{d}s\,.
\end{equation*}
and, upon change of variables $s = t_n + k\tau$ for $\tau \in [0,1]$, we have
\begin{equation}
\label{duhamel1a} 
\mathbf{W}(t_{n+1})= \mathrm{e}^{-k\mathbf{A}_3}\mathbf{W}(t_n) + k\int_0^1 \mathrm{e}^{-k\mathbf{A}_3(1-\tau)}\mathbf{h}(\mathbf{W}(t_n+k\tau))\;\mathrm{d}\tau\,.
\end{equation}
Set $\widehat{\mathbf{h}}(\tau) = \mathbf{h}(\mathbf{W}(t_n+k\tau))$ and note that the linear approximation of the non-linear function $\widehat{\mathbf{h}}(\tau)$ 
is given by $\widehat{\mathbf{h}}(0)+(\widehat{\mathbf{h}}(1)-\widehat{\mathbf{h}}(0))\tau$ 
for $\tau \in [0,1].$ Substituting this linear approximation into (\ref{duhamel1a}) yields
\begin{equation}
\label{duhamel2a}
\mathbf{W}(t_{n+1}) = \mathrm{e}^{-k\mathbf{A}_3}\mathbf{W}(t_n) +k\int_0^1 \mathrm{e}^{-k\mathbf{A}_3(1-\tau)}\;\mathrm{d}\tau \widehat{\mathbf{h}}(0)  + k\int_0^1 \mathrm{e}^{-k\mathbf{A}_3(1-\tau)}\tau\;\mathrm{d}\tau (\widehat{\mathbf{h}}(1)-\widehat{\mathbf{h}}(0)).
\end{equation}
Now, we rewrite the integrals in (\ref{duhamel2a}) using the following lemma.
\begin{lemma}\label{integrals1a}
Let $\mathbf{A}$ be an arbitrary non-singular square matrix, then
\begin{eqnarray}
 k\int_0^1 \mathrm{e}^{-k\mathbf{A}(1-\tau)}\;\mathrm{d}\tau &=& \mathbf{A}^{-1}(\mathbf{I}-\mathrm{e}^{-k\mathbf{A}})\,,\label{firstpart}\\
 k\int_0^1 \mathrm{e}^{-k\mathbf{A}(1-\tau)}\tau\;\mathrm{d}\tau &=& k^{-1}\mathbf{A}^{-2}(k\mathbf{A} - \mathbf{I}+ \mathrm{e}^{-k\mathbf{A}})\label{secondpart}\,.
\end{eqnarray}
\end{lemma}
\begin{proof}
We have
\begin{eqnarray*}
\frac{\mathrm{d}}{\mathrm{d}\tau}\left(\mathbf{A}^{-1}\mathrm{e}^{-(1-\tau)k\mathbf{A}}\right) = k\mathbf{A}\mathbf{A}^{-1}\mathrm{e}^{-(1-\tau)k\mathbf{A}}= k\mathrm{e}^{-(1-\tau)k\mathbf{A}}
\end{eqnarray*}
and therefore
\begin{eqnarray*}
k\int_0^1 \mathrm{e}^{-k\mathbf{A}(1-\tau)}\;\mathrm{d}\tau = \int_0^{1} \frac{\mathrm{d}}{\mathrm{d}\tau}\left(\mathbf{A}^{-1}
\mathrm{e}^{-(1-\tau)k\mathbf{A}}\;\mathrm{d}\tau \right)
&=& \mathbf{A}^{-1}(\mathbf{I}-\mathrm{e}^{-k\mathbf{A}}).
\end{eqnarray*} 
To show equation (\ref{secondpart}), we use the change of variable $s = (1-\tau)k$, which yields
\begin{eqnarray}
 k\int_0^1 \mathrm{e}^{-k\mathbf{A}(1-\tau)}\tau\;\mathrm{d}\tau = \int_0^k \mathrm{e}^{-s\mathbf{A}}(1-k^{-1}s)\;\mathrm{d}s
= \int_0^k \mathrm{e}^{-s\mathbf{A}}\;\mathrm{d}s - k^{-1}\int_0^k \mathrm{e}^{-s\mathbf{A}}s\;\mathrm{d}s\,.
\label{step}
\end{eqnarray} 
The first integral of the right-hand side of (\ref{step}) can be transformed to
\begin{align*}
\int_0^k \mathrm{e}^{-s\mathbf{A}}\;\mathrm{d}s = -\int_0^k \frac{\mathrm{d}}{\mathrm{d}s}(\mathbf{A}^{-1}\mathrm{e}^{-s\mathbf{A}})\;\mathrm{d}s = \mathbf{A}^{-1}(\mathbf{I}-\mathrm{e}^{-k\mathbf{A}})
\end{align*}
using 
\[\frac{\mathrm{d}}{\mathrm{d}s}(\mathbf{A}^{-1}\mathrm{e}^{-s\mathbf{A}}) = -\mathbf{A}^{-1}\mathbf{A}\mathrm{e}^{-s\mathbf{A}} = -\mathrm{e}^{-s\mathbf{A}}\,.\]
We have
\[\frac{\mathrm{d}}{\mathrm{d}s}(k^{-1}s\mathbf{A}^{-1}\mathrm{e}^{-s\mathbf{A}}) = k^{-1}\mathbf{A}^{-1}\mathrm{e}^{-s\mathbf{A}}-k^{-1}\mathrm{e}^{-s\mathbf{A}}s\]
and hence, we obtain for the second integral of the right-hand side of (\ref{step})
\begin{eqnarray*}
k^{-1}\int_{0}^k \mathrm{e}^{-s\mathbf{A}}s\;\mathrm{d}s &=& k^{-1}\int_0^k \mathbf{A}^{-1}\mathrm{e}^{-s\mathbf{A}}\;\mathrm{d}s - \int_0^k \frac{\mathrm{d}}{\mathrm{d}s}(k^{-1}s\mathbf{A}^{-1}\mathrm{e}^{-s\mathbf{A}})\;\mathrm{d}s\\
&=& k^{-1}\mathbf{A}^{-2}\left(\mathbf{I}-\mathrm{e}^{-k\mathbf{A}}\right)-\mathbf{A}^{-1}\mathrm{e}^{-k\mathbf{A}}\\
&=& k^{-1}\mathbf{A}^{-2}\left(\mathbf{I}-\mathrm{e}^{-k\mathbf{A}}-k\mathbf{A}\mathrm{e}^{-k\mathbf{A}}\right)\,.
\end{eqnarray*}
Thus, equation (\ref{step}) can be written as
\begin{eqnarray*}
k\int_0^1 \mathrm{e}^{-k\mathbf{A}(1-\tau)}\tau\;\mathrm{d}\tau &=& \mathbf{A}^{-1}(\mathbf{I}-\mathrm{e}^{-k\mathbf{A}})-k^{-1}\mathbf{A}^{-2}
\left(\mathbf{I}-\mathrm{e}^{-k\mathbf{A}}-k\mathbf{A}\mathrm{e}^{-k\mathbf{A}}\right)\\
&=&k^{-1}\mathbf{A}^{-2}\left(k\mathbf{A} -\mathbf{I}+\mathrm{e}^{-k\mathbf{A}}\right)\,.
\end{eqnarray*}
\end{proof}
 Applying the results of Lemma \ref{integrals1a} to (\ref{duhamel2a}), we obtain
\begin{eqnarray}
\label{duhamel3}
\mathbf{W}(t_{n+1}) &=& \mathrm{e}^{-k\mathbf{A}_3}\mathbf{W}(t_n) + \mathbf{A}_3^{-1}(\mathbf{I}_m-\mathrm{e}^{-k\mathbf{A}_3})\mathbf{h}(\mathbf{W}(t_n)) \nonumber\\
&+& k^{-1}\mathbf{A}_3^{-2}(k\mathbf{A}_3 - \mathbf{I}_m+ 
\mathrm{e}^{-k\mathbf{A}_3})(\mathbf{h}(\mathbf{W}(t_{n+1}))-\mathbf{h}(\mathbf{W}(t_n)))
\end{eqnarray}
which we write as
\begin{eqnarray}
\label{duhamel3z}
\mathbf{W}_{n+1} &=& \mathrm{e}^{-k\mathbf{A}_3}\mathbf{W}_n + \mathbf{A}_3^{-1}(\mathbf{I}_m-\mathrm{e}^{-k\mathbf{A}_3})\mathbf{h}(\mathbf{W}_n) \nonumber\\
&+& k^{-1}\mathbf{A}_3^{-2}(k\mathbf{A}_3 - \mathbf{I}_m+ 
\mathrm{e}^{-k\mathbf{A}_3})(\mathbf{h}(\mathbf{W}_{n+1})-\mathbf{h}(\mathbf{W}_n))\,.
\end{eqnarray}
Note, however that the current scheme (\ref{duhamel3z}) is fully implicit and hence one would need to employ Newton-type solvers. We therefore use the (easy to derive) first order approximation of (\ref{duhamel3z}) given by 
\begin{eqnarray}
\label{duhamel3t}
\mathbf{W}_{n+1}^\ast &=& \mathrm{e}^{-k\mathbf{A}_3}\mathbf{W}_n + \mathbf{A}_3^{-1}(\mathbf{I}_m-\mathrm{e}^{-k\mathbf{A}_3})\mathbf{h}(\mathbf{W}_n)
\end{eqnarray}
known as ETD1 (exponential time differencing of order one which is the same as the IMEX backward Euler method).
Next, we set $\mathbf{h}(\mathbf{W}_{n+1})=\mathbf{h}(\mathbf{W}_{n+1}^\ast)$. Hence, our second-order semi-discrete exponential time differencing (ETD) scheme is (see also p. 25 in \cite{asante})
\begin{eqnarray}
\label{duhamel3u}
\mathbf{W}_{n+1} &=& \mathrm{e}^{-k\mathbf{A}_3}\mathbf{W}_n + \mathbf{A}_3^{-1}(\mathbf{I}_m-\mathrm{e}^{-k\mathbf{A}_3})\mathbf{h}(\mathbf{W}_n) \nonumber\\
&+& k^{-1}\mathbf{A}_3^{-2}(k\mathbf{A}_3 - \mathbf{I}_m+ 
\mathrm{e}^{-k\mathbf{A}_3})(\mathbf{h}(\mathbf{W}_{n+1}^\ast)-\mathbf{h}(\mathbf{W}_n))\nonumber\\
\mathbf{W}_{n+1}^\ast &=& \mathrm{e}^{-k\mathbf{A}_3}\mathbf{W}_n + \mathbf{A}_3^{-1}(\mathbf{I}_m-\mathrm{e}^{-k\mathbf{A}_3})\mathbf{h}(\mathbf{W}_n)\,.
\end{eqnarray}
Finally, the unwinding back to $\mathbf{U}$ is needed. However, we explain this later, since we first concentrate on the approximation of the matrix exponential.

\section{Approximating the matrix exponential}\label{section6}
The final step is to approximate the matrix exponential. There are several ways to achieve this (see \cites{Kleefeld2012,Yousuf2012,moler1978,moler}). However, we focus on the second-order $L$-acceptable rational 
approximation with simple real distinct poles (RDP) originally proposed in 
\cite{Khaliq1994} and used in \cites{asantethesis,asante} to derive the ETD-RDP scheme. A second-order approximation of the matrix exponential using rational approximation \cite{asante}*{p. 26} 
is given by 
\begin{eqnarray}
 \mathbf{R}_{\mathrm{RDP}}(k\mathbf{A_3})=\left(\mathbf{I}_m-\frac{5k}{12}\mathbf{A}_3\right)\left[\left(\mathbf{I}_m+\frac{k}{3}\mathbf{A}_3\right)\left(\mathbf{I}_m+\frac{k}{4}\mathbf{A}_3\right)\right]^{-1}\approx\mathrm{e}^{-k\mathbf{A_3}}
 \label{padenonlinear}
\end{eqnarray}
and the Pad\'e-$(0,1)$ first-order approximation is given by
\begin{eqnarray}
 \mathbf{R}_{\mathrm{01}}(k\mathbf{A_3})=\left(\mathbf{I}_m+k\mathbf{A}_3\right)^{-1}\approx\mathrm{e}^{-k\mathbf{A_3}}\,.
 \label{pade01}
\end{eqnarray}
Using (\ref{padenonlinear}) and (\ref{pade01}) for (\ref{duhamel3u}) yields
\begin{eqnarray}
\label{duhamel4u}
\mathbf{W}_{n+1} &=&  \mathbf{R}_{\mathrm{RDP}}(k\mathbf{A_3})\mathbf{W}_n + \mathbf{A}_3^{-1}(\mathbf{I}_m- \mathbf{R}_{\mathrm{RDP}}(k\mathbf{A_3}))\mathbf{h}(\mathbf{W}_n) \nonumber\\
&+& k^{-1}\mathbf{A}_3^{-2}(k\mathbf{A}_3 - \mathbf{I}_m+ 
 \mathbf{R}_{\mathrm{RDP}}(k\mathbf{A_3}))(\mathbf{h}(\mathbf{W}_{n+1}^\ast)-\mathbf{h}(\mathbf{W}_n))\nonumber\\
\mathbf{W}_{n+1}^\ast &=& \mathbf{R}_{\mathrm{01}}(k\mathbf{A_3})\mathbf{W}_n + \mathbf{A}_3^{-1}(\mathbf{I}_m-\mathbf{R}_{\mathrm{01}}(k\mathbf{A_3}))\mathbf{h}(\mathbf{W}_n)\,.
\end{eqnarray}
The first-order predictor $\mathbf{W}_{n+1}^\ast$ can be simplified to
\begin{eqnarray*}
 \mathbf{W}_{n+1}^\ast&=& \mathbf{R}_{\mathrm{01}}(k\mathbf{A_3})\mathbf{W}_n+\mathbf{A}_3^{-1}\left(\mathbf{R}_{\mathrm{01}}(k\mathbf{A_3})^{-1}\mathbf{R}_{\mathrm{01}}(k\mathbf{A_3})-\mathbf{R}_{\mathrm{01}}(k\mathbf{A_3})\right)\mathbf{h}(\mathbf{W}_n)\\
 &=&\mathbf{R}_{\mathrm{01}}(k\mathbf{A_3})\mathbf{W}_n+\mathbf{A}_3^{-1}(\mathbf{I}_m+k\mathbf{A}_3-\mathbf{I}_m)\mathbf{R}_{\mathrm{01}}(k\mathbf{A_3})\mathbf{h}(\mathbf{W}_n)\\
 &=&\mathbf{R}_{\mathrm{01}}(k\mathbf{A_3})\left(\mathbf{W}_n+k \mathbf{h}(\mathbf{W}_n)\right)\\
 &=&\left(\mathbf{I}_m+k\mathbf{A}_3\right)^{-1}\left(\mathbf{W}_n+k\mathbf{h}(\mathbf{W}_n)\right)\,.
\end{eqnarray*}
Next, we focus on the corrector $\mathbf{W}_{n+1}$. We have
\begin{eqnarray*}
 & &\mathbf{A}_3^{-1}(\mathbf{I}_m- \mathbf{R}_{\mathrm{RDP}}(k\mathbf{A_3}))\\
 &=&\mathbf{A}_3^{-1}\left(\mathbf{I}_m- \left(\mathbf{I}_m-\frac{5k}{12}\mathbf{A}_3\right)\left[\left(\mathbf{I}_m+\frac{k}{3}\mathbf{A}_3\right)\left(\mathbf{I}_m+\frac{k}{4}\mathbf{A}_3\right)\right]^{-1}\right)\\
 &=&\mathbf{A}_3^{-1}\left(\left(\mathbf{I}_m+\frac{k}{3}\mathbf{A}_3\right)\left(\mathbf{I}_m+\frac{k}{4}\mathbf{A}_3\right)- \left(\mathbf{I}_m-\frac{5k}{12}\mathbf{A}_3\right)\right) \\
 &\bullet&\left(\mathbf{I}_m+\frac{k}{4}\mathbf{A}_3\right)^{-1}\left(\mathbf{I}_m+\frac{k}{3}\mathbf{A}_3\right)^{-1}\\
 &=& \mathbf{A}_3^{-1}\left(k\mathbf{A}_3 +\frac{k^2}{12}\mathbf{A}_3^2\right) \left(\mathbf{I}_m+\frac{k}{4}\mathbf{A}_3\right)^{-1}\left(\mathbf{I}_m+\frac{k}{3}\mathbf{A}_3\right)^{-1}\\
 &=& k\left(\mathbf{I}_m +\frac{k}{12}\mathbf{A}_3\right) \left(\mathbf{I}_m+\frac{k}{4}\mathbf{A}_3\right)^{-1}\left(\mathbf{I}_m+\frac{k}{3}\mathbf{A}_3\right)^{-1}
\end{eqnarray*}
and likewise 
\begin{eqnarray*}
 & & k^{-1}\mathbf{A}_3^{-2}(k\mathbf{A}_3 - \mathbf{I}_m+ \mathbf{R}_{\mathrm{RDP}}(k\mathbf{A_3}))\\
 &=& k^{-1}\mathbf{A}_3^{-2}\left[ k \mathbf{A}_3 \left(\mathbf{I}_m+\frac{k}{3}\mathbf{A}_3\right)\left(\mathbf{I}_m+\frac{k}{4}\mathbf{A}_3\right)    -  \left(\mathbf{I}_m+\frac{k}{3}\mathbf{A}_3\right)\left(\mathbf{I}_m+\frac{k}{4}\mathbf{A}_3\right)\right.\\
 &+&\left. \left(\mathbf{I}_m-\frac{5k}{12}\mathbf{A}_3\right)\right] \left(\mathbf{I}_m+\frac{k}{4}\mathbf{A}_3\right)^{-1}\left(\mathbf{I}_m+\frac{k}{3}\mathbf{A}_3\right)^{-1}\\
 &=&k^{-1}\mathbf{A}_3^{-2}\left[\frac{6k^2}{12}\mathbf{A}_3^2+\frac{k^3}{12}\mathbf{A}_3^3\right] \left(\mathbf{I}_m+\frac{k}{4}\mathbf{A}_3\right)^{-1}\left(\mathbf{I}_m+\frac{k}{3}\mathbf{A}_3\right)^{-1}\\
 &=&\frac{k}{2}\left(\mathbf{I}_m+\frac{k}{6}\mathbf{A}_3\right) \left(\mathbf{I}_m+\frac{k}{4}\mathbf{A}_3\right)^{-1}\left(\mathbf{I}_m+\frac{k}{3}\mathbf{A}_3\right)^{-1}\,.
\end{eqnarray*}
The corrector can now be written as
\begin{eqnarray*}
\mathbf{W}_{n+1} &=& \left(\mathbf{I}_m-\frac{5k}{12}\mathbf{A}_3\right)\left(\mathbf{I}_m+\frac{k}{4}\mathbf{A}_3\right)^{-1}\left(\mathbf{I}_m+\frac{k}{3}\mathbf{A}_3\right)^{-1} \mathbf{W}_n \\
&+& \frac{k}{2}\left(\mathbf{I}_m+\frac{k}{4}\mathbf{A}_3\right)^{-1}\left(\mathbf{I}_m+\frac{k}{3}\mathbf{A}_3\right)^{-1} \mathbf{h}(\mathbf{W}_n) \\
&+& \frac{k}{2}\left(\mathbf{I}_m+\frac{k}{6}\mathbf{A}_3\right) \left(\mathbf{I}_m+\frac{k}{4}\mathbf{A}_3\right)^{-1}\left(\mathbf{I}_m+\frac{k}{3}\mathbf{A}_3\right)^{-1}\mathbf{h}(\mathbf{W}_{n+1}^\ast)\,.
\end{eqnarray*}
These three terms can be nicely rewritten using the following three partial fraction decompositions (cf. \cite{asantethesis}*{p. 32} for the details)
\begin{eqnarray*}
\left(\mathbf{I}_m-\frac{5k}{12}\mathbf{A}_3\right)\mathbf{S}_{m,k}\mathbf{T}_{m,k}
&=&9\mathbf{T}_{m,k}-8\mathbf{S}_{m,k}\\
\mathbf{S}_{m,k}\mathbf{T}_{m,k}
&=&4\mathbf{T}_{m,k}-3\mathbf{S}_{m,k}\\
 \left(\mathbf{I}_m+\frac{k}{6}\mathbf{A}_3\right) \mathbf{S}_{m,k}\mathbf{T}_{m,k}
&=&2\mathbf{T}_{m,k}-1 \mathbf{S}_{m,k}\,,
\end{eqnarray*}
where we used the abbreviations
\begin{eqnarray*}
 \mathbf{S}_{m,k}&=&\left(\mathbf{I}_m+\frac{k}{4}\mathbf{A}_3\right)^{-1}\\
 \mathbf{T}_{m,k}&=&\left(\mathbf{I}_m+\frac{k}{3}\mathbf{A}_3\right)^{-1}
\end{eqnarray*}
to shorten the notation for these decompositions. We have
\begin{eqnarray*}
 \mathbf{W}_{n+1}&=&\left(9\mathbf{T}_{m,k}-8\mathbf{S}_{m,k}\right)\mathbf{W}_n+ \frac{k}{2}\left(4\mathbf{T}_{m,k}-3\mathbf{S}_{m,k}\right)\mathbf{h}(\mathbf{W}_n)\\
&+&\frac{k}{2}\left(2\mathbf{T}_{m,k}-\mathbf{S}_{m,k}\right)\mathbf{h}(\mathbf{W}_{n+1}^\ast)
\end{eqnarray*}
which can be written as
\begin{eqnarray*}
\mathbf{W}_{n+1}&=&\mathbf{T}_{m,k}\left(9\mathbf{W}_n+2k\mathbf{h}(\mathbf{W}_n)+ k\mathbf{h}(\mathbf{W}_{n+1}^\ast)\right)\\
&-&\mathbf{S}_{m,k}\left(8\mathbf{W}_n+\frac{3k}{2}\mathbf{h}(\mathbf{W}_n)+\frac{k}{2}\mathbf{h}(\mathbf{W}_{n+1}^\ast)\right)
 \end{eqnarray*}
Hence, the fully discrete ETD-RDP scheme is given by
\begin{eqnarray}
\mathbf{W}_{n+1} &=& \left(\mathbf{I}_m+\frac{k}{3}\mathbf{A}_3\right)^{-1}\left[9\mathbf{W}_n+2k\mathbf{h}(\mathbf{W}_n)+k \mathbf{h}(\mathbf{W}_{n+1}^\ast)  \right]\nonumber\\
&-& \left(\mathbf{I}_m+\frac{k}{4}\mathbf{A}_3\right)^{-1}\left[8\mathbf{W}_n+\frac{3k}{2}\mathbf{h}(\mathbf{W}_n)+\frac{k}{2} \mathbf{h}(\mathbf{W}_{n+1}^\ast)  \right]\nonumber\\
\mathbf{W}_{n+1}^\ast&=&\left(\mathbf{I}_m+k\mathbf{A}_3\right)^{-1}\left(\mathbf{W}_n+k\mathbf{h}(\mathbf{W}_n)\right)\,.
\label{finalle}
\end{eqnarray}
Note that the stability regions of this second-order $L$-stable scheme 
(see \cite{asantethesis}*{Theorems 4.0.2--4.0.4} for the detailed proof of second order accuracy and \cite{asante}*{p. 26} for the $L$-stability, respectively) are given in 
\cite{asantethesis}*{Figure 2.2} and compared against ETD-CN, ETD-Pad\'e-$(0,2)$, and implicit-explicit Adams-Moulton/Bashforth scheme.

Now, we substitute $\mathbf{W}=\mathrm{e}^{\mathbf{A}_2 t}\mathbf{V}$ and $\mathbf{h}(\mathbf{W})=\mathrm{e}^{\mathbf{A}_2 t}\mathbf{g}(\mathbf{V})$, which means we have the expressions 
$\mathbf{W}_n=\mathrm{e}^{\mathbf{A}_2nk}\mathbf{V}_n$, $\mathbf{W}_{n+1}=\mathrm{e}^{\mathbf{A}_2nk}\mathrm{e}^{\mathbf{A}_2k}\mathbf{V}_{n+1}$, $\mathbf{h}(\mathbf{W}_n)=\mathrm{e}^{\mathbf{A}_2 nk}\mathbf{g}(\mathbf{V}_n)$, 
and further
$\mathbf{h}(\mathbf{W}_{n+1})=\mathrm{e}^{\mathbf{A}_2 nk}\mathrm{e}^{\mathbf{A}_2k}\mathbf{g}(\mathbf{V}_{n+1})$.
Then, (\ref{finalle}) can be written as
\begin{eqnarray*}
\mathbf{V}_{n+1} &=& \left(\mathbf{I}_m+\frac{k}{3}\mathbf{A}_3\right)^{-1}\left[\mathrm{e}^{-\mathbf{A}_2k}\left\{9\mathbf{V}_n+2k\mathbf{h}(\mathbf{V}_n)\right\}+k \mathbf{g}(\mathbf{V}_{n+1}^\ast)  \right]\\
&-& \left(\mathbf{I}_m+\frac{k}{4}\mathbf{A}_3\right)^{-1}\left[\mathrm{e}^{-\mathbf{A}_2k}\left\{8\mathbf{V}_n+\frac{3k}{2}\mathbf{h}(\mathbf{V}_n)\right\}+\frac{k}{2} \mathbf{h}(\mathbf{V}_{n+1}^\ast)  \right]\\
\mathbf{V}_{n+1}^\ast&=&\left(\mathbf{I}_m+k\mathbf{A}_3\right)^{-1}\mathrm{e}^{-\mathbf{A}_2k}\left(\mathbf{V}_n+k\mathbf{h}(\mathbf{V}_n)\right)\,.
\end{eqnarray*}
Now, we approximate $\mathrm{e}^{-\mathbf{A}_2k}$ in the corrector with $\mathbf{R}_{\mathrm{RDP}}$ and in the predictor with $\mathbf{R}_{\mathrm{01}}$. This gives
\begin{eqnarray}
\mathbf{V}_{n+1} &=& \left(\mathbf{I}_m+\frac{k}{3}\mathbf{A}_3\right)^{-1}\left[\left\{9\left(\mathbf{I}_m+\frac{k}{3}\mathbf{A}_2\right)^{-1}-8\left(\mathbf{I}_m+\frac{k}{4}\mathbf{A}_2\right)^{-1}\right\}    \right.\nonumber\\
&\bullet&\left.\left\{9\mathbf{V}_n+2k\mathbf{g}(\mathbf{V}_n)\right\}+k \mathbf{g}(\mathbf{V}_{n+1}^\ast)  \right]\nonumber\\
&-& \left(\mathbf{I}_m+\frac{k}{4}\mathbf{A}_3\right)^{-1}\left[\left\{9\left(\mathbf{I}_m+\frac{k}{3}\mathbf{A}_2\right)^{-1}-8\left(\mathbf{I}_m+\frac{k}{4}\mathbf{A}_2\right)^{-1}\right\}                     \right.\nonumber\\
&\bullet&\left.\left\{8\mathbf{V}_n+\frac{3k}{2}\mathbf{g}(\mathbf{V}_n)\right\}+\frac{k}{2} \mathbf{g}(\mathbf{V}_{n+1}^\ast)  \right]\nonumber\\
\mathbf{V}_{n+1}^\ast&=&\left(\mathbf{I}_m+k\mathbf{A}_3\right)^{-1}\left(\mathbf{I}_m+k\mathbf{A}_2\right)^{-1}\left(\mathbf{V}_n+k\mathbf{g}(\mathbf{V}_n)\right)\,.
\label{done}
\end{eqnarray}
Next, we recall the substitution $\mathbf{V}=\mathrm{e}^{\mathbf{A}_1 t}\mathbf{U}$ and $\mathbf{g}(\mathbf{V})=\mathrm{e}^{\mathbf{A}_1 t}\mathbf{f}(\mathbf{U})$, which means we have
$\mathbf{V}_n=\mathrm{e}^{\mathbf{A}_1nk}\mathbf{U}_n$, $\mathbf{V}_{n+1}=\mathrm{e}^{\mathbf{A}_1nk}\mathrm{e}^{\mathbf{A}_1k}\mathbf{U}_{n+1}$, $\mathbf{g}(\mathbf{V}_n)=\mathrm{e}^{\mathbf{A}_1 nk}\mathbf{f}(\mathbf{U}_n)$, 
and 
$\mathbf{g}(\mathbf{V}_{n+1})=\mathrm{e}^{\mathbf{A}_1 nk}\mathrm{e}^{\mathbf{A}_1k}\mathbf{f}(\mathbf{U}_{n+1})$. We approximate
$\mathrm{e}^{-\mathbf{A}_1k}$ in the corrector with $\mathbf{R}_{\mathrm{RDP}}$ and in the predictor with $\mathbf{R}_{\mathrm{01}}$, which finally yields the fully discrete iterative scheme given in the following 
frame.

\begin{mdframed}
\textbf{ETD-RDP-IF scheme in three dimensions}\\
For $n=0,1,2,\ldots, T/k-1$ compute
\begin{eqnarray}
\mathbf{U}_{n+1} &=& \left(\mathbf{I}_m+\frac{k}{3}\mathbf{A}_3\right)^{-1}\left[\left\{9\left(\mathbf{I}_m+\frac{k}{3}\mathbf{A}_2\right)^{-1}-8\left(\mathbf{I}_m+\frac{k}{4}\mathbf{A}_2\right)^{-1}\right\}    \right.\nonumber\\
&\bullet& \left.\left\{9\left(\mathbf{I}_m+\frac{k}{3}\mathbf{A}_1\right)^{-1}-8\left(\mathbf{I}_m+\frac{k}{4}\mathbf{A}_1\right)^{-1}\right\}\left\{9\mathbf{U}_n+2k\mathbf{f}(\mathbf{U}_n)\right\}+k \mathbf{f}(\mathbf{U}_{n+1}^\ast)  \right]\nonumber\\
&-& \left(\mathbf{I}_m+\frac{k}{4}\mathbf{A}_3\right)^{-1}\left[\left\{9\left(\mathbf{I}_m+\frac{k}{3}\mathbf{A}_2\right)^{-1}-8\left(\mathbf{I}_m+\frac{k}{4}\mathbf{A}_2\right)^{-1}\right\}                     \right.\nonumber\\
&\bullet&\left.\left\{9\left(\mathbf{I}_m+\frac{k}{3}\mathbf{A}_1\right)^{-1}-8\left(\mathbf{I}_m+\frac{k}{4}\mathbf{A}_1\right)^{-1}\right\}\left\{8\mathbf{U}_n+\frac{3k}{2}\mathbf{f}(\mathbf{U}_n)\right\}+\frac{k}{2} \mathbf{f}(\mathbf{U}_{n+1}^\ast)  \right]\nonumber\\
\mathbf{U}_{n+1}^\ast&=&\left(\mathbf{I}_m+k\mathbf{A}_3\right)^{-1}\left(\mathbf{I}_m+k\mathbf{A}_2\right)^{-1}\left(\mathbf{I}_m+k\mathbf{A}_1\right)^{-1}\left(\mathbf{U}_n+k\mathbf{f}(\mathbf{U}_n)\right)\nonumber\\
\label{donedone}
\end{eqnarray}
with $\mathbf{U}_0=\mathbf{U}(0)$.
\end{mdframed}

This procedure easily extends to arbitrary dimensions $d$. We derive the fully discrete iterative ETD-RDP-IF scheme in $d$ dimensions. It is given in the following frame.
\begin{mdframed}
\textbf{ETD-RDP-IF scheme in $d$ dimensions}\\
For $n=0,1,2,\ldots, T/k-1$ compute
\begin{eqnarray*}
\mathbf{U}_{n+1} &=& \left(\mathbf{I}_m+\frac{k}{3}\mathbf{A}_d\right)^{-1}\left[\left\{9\left(\mathbf{I}_m+\frac{k}{3}\mathbf{A}_{d-1}\right)^{-1}-8\left(\mathbf{I}_m+\frac{k}{4}\mathbf{A}_{d-1}\right)^{-1}\right\}\bullet\ldots    \right.\nonumber\\
&\bullet& \left.\left\{9\left(\mathbf{I}_m+\frac{k}{3}\mathbf{A}_1\right)^{-1}-8\left(\mathbf{I}_m+\frac{k}{4}\mathbf{A}_1\right)^{-1}\right\}\left\{9\mathbf{U}_n+2k\mathbf{f}(\mathbf{U}_n)\right\}+k \mathbf{f}(\mathbf{U}_{n+1}^\ast)  \right]\nonumber\\
&-& \left(\mathbf{I}_m+\frac{k}{4}\mathbf{A}_d\right)^{-1}\left[\left\{9\left(\mathbf{I}_m+\frac{k}{3}\mathbf{A}_{d-1}\right)^{-1}-8\left(\mathbf{I}_m+\frac{k}{4}\mathbf{A}_{d-1}\right)^{-1}\right\}\bullet\ldots                     \right.\nonumber\\
&\bullet&\left.\left\{9\left(\mathbf{I}_m+\frac{k}{3}\mathbf{A}_1\right)^{-1}-8\left(\mathbf{I}_m+\frac{k}{4}\mathbf{A}_1\right)^{-1}\right\}\left\{8\mathbf{U}_n+\frac{3k}{2}\mathbf{f}(\mathbf{U}_n)\right\}+\frac{k}{2} \mathbf{f}(\mathbf{U}_{n+1}^\ast)  \right]\nonumber\\
\mathbf{U}_{n+1}^\ast&=&\left(\mathbf{I}_m+k\mathbf{A}_d\right)^{-1}\bullet\ldots\bullet\left(\mathbf{I}_m+k\mathbf{A}_1\right)^{-1}\left(\mathbf{U}_n+k\mathbf{f}(\mathbf{U}_n)\right)\,.
\end{eqnarray*}
with $\mathbf{U}_0=\mathbf{U}(0)$.
\end{mdframed}
\section{Implementation and parallelization of the fully discrete scheme}\label{section7}
An implementation of this scheme given by (\ref{donedone}) in parallel utilizing only three threads for the three-dimensional case is illustrated in the flow chart given in Figure \ref{flow3d}. One could theoretically 
also use five threads 
to speed up the computations. However, we use for the numerical results
a computer which has at most four threads as explained in the next section. 
The two-dimensional case is illustrated in Figure \ref{flow2d} and can be derived using the system given in (\ref{done}) replacing $\mathbf{W}$ with $\mathbf{U}$, $\mathbf{g}$ with $\mathbf{f}$, $\mathbf{A}_2$ with $\mathbf{A}_1$ and $\mathbf{A}_3$ 
with $\mathbf{A}_2$.
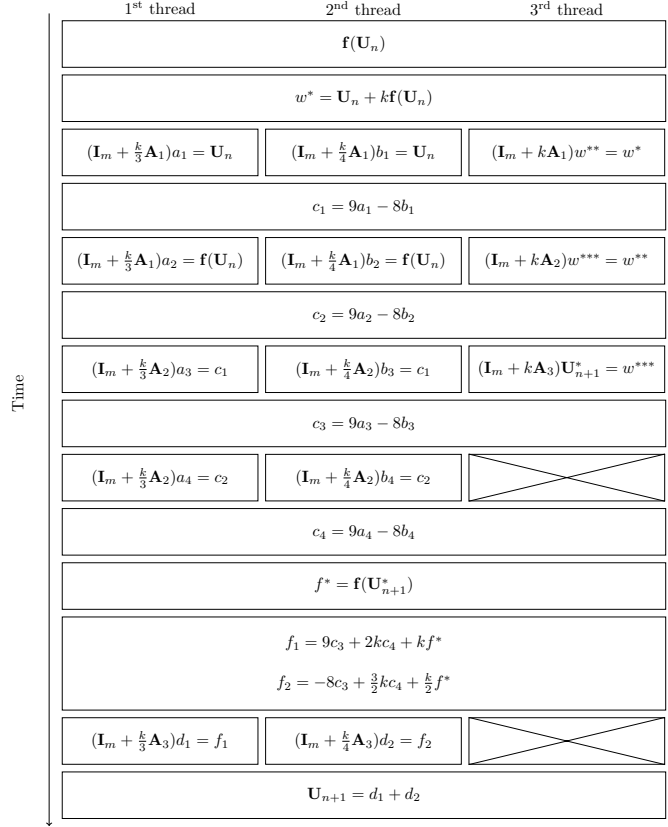
\begin{figure}[!ht]
\centering
\resizebox{9cm}{!}{%
\begin{tikzpicture}[rectangle,text centered,node distance=0.15cm,minimum height=1cm]
    \node (z01) [minimum height=0cm] {}; 
    \node (z011) [right=of z01, minimum width=\schmal cm]{1$^{\mathrm{st}}$ thread};
    \node (z012) [right=of z011,minimum width=\schmal cm]{2$^{\mathrm{nd}}$ thread};
    \node (z013) [right=of z012,minimum width=\schmal cm]{3$^{\mathrm{rd}}$ thread};
    
    \node (z02) [below=of z01] {};
    \node (z021) [draw,right=of z02, minimum width=\breite cm] 	{{$\mathbf{f}(\mathbf{U}_{n})$}};

    \node (z03) [below=of z02] {};
    \node (z031) [draw,right=of z03, minimum width=\breite cm]	{{$ w^\ast=\mathbf{U}_{n} + k\mathbf{f}(\mathbf{U}_{n}) $}};

    \node (z04) [below=of z03] {};
    \node (z041) [draw,right=of z04, minimum width=\schmal cm] 	{$ (\mathbf{I}_m+ \frac{k}{3} \mathbf{A}_1) a_{1} = \mathbf{U}_{n} $};
    \node (z042) [draw,right=of z041, minimum width=\schmal cm]	{$ (\mathbf{I}_m+ \frac{k}{4} \mathbf{A}_1) b_{1} = \mathbf{U}_{n} $};
    \node (z043) [draw,right=of z042,minimum width=\schmal cm] 	{$ (\mathbf{I}_m+ k \mathbf{A}_1) w^{\ast\ast} = w^\ast$};

    \node (z05) [below=of z04] {};
    \node (z051) [draw,right=of z05, minimum width=\breite cm]{$ c_1 = 9a_1 - 8b_1$};
    
    \node (z06) [below=of z05] {};
    \node (z061) [draw,right=of z06, minimum width=\schmal cm]{$ (\mathbf{I}_m+ \frac{k}{3} \mathbf{A}_1) a_{2} = \mathbf{f}(\mathbf{U}_{n}) $};
    \node (z062) [draw,right=of z061,minimum width=\schmal cm]{$ (\mathbf{I}_m+ \frac{k}{4} \mathbf{A}_1) b_{2} = \mathbf{f}(\mathbf{U}_{n}) $};
    \node (z063) [draw,right=of z062,minimum width=\schmal cm]{$ (\mathbf{I}_m+ k \mathbf{A}_2)w^{\ast\ast\ast} = w^{\ast\ast} $};
    
    \node (z07) [below=of z06] {};
    \node (z071) [draw,right=of z07, minimum width=\breite cm]{$ c_2 = 9a_2 - 8b_2 $};
    
    \node (z08) [below=of z07] {};
    \node (z081) [draw,right=of z08, minimum width=\schmal cm]{$ (\mathbf{I}_m+ \frac{k}{3} \mathbf{A}_2) a_{3} = c_1 $};
    \node (z082) [draw,right=of z081,minimum width=\schmal cm]{$ (\mathbf{I}_m+ \frac{k}{4} \mathbf{A}_2) b_{3} = c_1 $};
    \node (z083) [draw,right=of z082,minimum width=\schmal cm]{$ (\mathbf{I}_m+ k \mathbf{A}_3)\mathbf{U}^\ast_{n+1} = w^{\ast\ast\ast} $};
    
    \node (z09) [below=of z08] {};
    \node (z091) [draw,right=of z09,minimum width=\breite cm]{$ c_3 = 9a_3 - 8b_3 $};
    
    \node (z10) [below=of z09] {};
    \node (z101) [draw,right=of z10, minimum width=\schmal cm]{$ (\mathbf{I}_m+ \frac{k}{3} \mathbf{A}_2) a_{4} = c_2 $};
    \node (z102) [draw,right=of z101,minimum width=\schmal cm]{$ (\mathbf{I}_m+ \frac{k}{4} \mathbf{A}_2) b_{4} = c_2 $};
    \node (z103) [draw,right=of z102,minimum width=\schmal cm]{};
    \node (z103c)[draw,cross out,right=of z102,minimum width=\schmal cm]{};

    \node (z11) [below=of z10] {};
    \node (z111) [draw,right=of z11,minimum width=\breite cm]{$ c_4 = 9a_4 - 8b_4 $};

    \node (z08a) [below=of z11] {};
    \node (z081a) [draw,right=of z08a, minimum width=\breite cm]{$ f^\ast=\mathbf{f}(\mathbf{U}^\ast_{n+1}) $};
    
    \node (z12) [below=of z08a,minimum height=2cm] {};
    \node (z121) [draw,right=of z12,minimum width=\breite cm,minimum height=2cm]{\begin{tabular}{cc}
										  $ f_1 = 9 c_3 + 2 k c_4 + k f^\ast $ \\ \\ 
										  $ f_2 = -8 c_3 + \frac{3}{2} k c_4 + \frac{k}{2} f^\ast $ \\
										  \end{tabular}};
    
    \node (z13) [below=of z12] {};
    \node (z131) [draw,right=of z13, minimum width=\schmal cm]{$ (\mathbf{I}_m+ \frac{k}{3} \mathbf{A}_3) d_{1} = f_1 $};
    \node (z132) [draw,right=of z131,minimum width=\schmal cm]{$ (\mathbf{I}_m+ \frac{k}{4} \mathbf{A}_3) d_{2} = f_2 $};
    \node (z133) [draw,right=of z132,minimum width=\schmal cm]{};
    \node (z133c)[draw,cross out,right=of z132,minimum width=\schmal cm]{};
    
    \node (z14) [below=of z13] {};
    \node (z141) [draw,right=of z14,minimum width=\breite cm]{$ \mathbf{U}_{n+1} = d_1 + d_2 $};
    
    \node (z15) [below=of z14] {};

    \path [draw,->] (z01) -- (z15);
    \node (pfeil1) [left=of z08]{};
    \node (pfeil2) [left=of pfeil1,rotate=90]{Time};
  \end{tikzpicture}
  }
  \caption{\label{flow3d}Flow chart for the implementation of the ETD-RDP-IF scheme in parallel using three threads for the three dimensional case.}
\end{figure}
 
 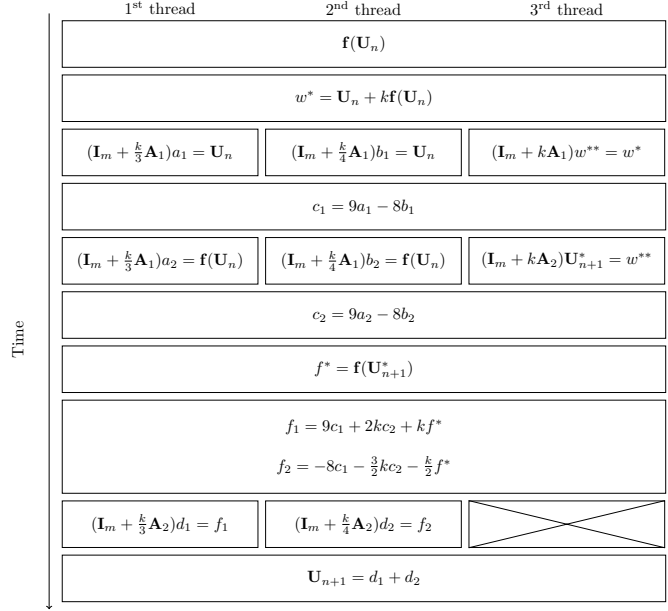
\begin{figure}[!ht]
 \centering
\resizebox{9cm}{!}{%
 \begin{tikzpicture}[rectangle,text centered,node distance=0.15cm,minimum height=1cm]
    \node (z01) [minimum height=0cm] {}; 
    \node (z011) [right=of z01, minimum width=\schmal cm]{1$^{\mathrm{st}}$ thread};
    \node (z012) [right=of z011,minimum width=\schmal cm]{2$^{\mathrm{nd}}$ thread};
    \node (z013) [right=of z012,minimum width=\schmal cm]{3$^{\mathrm{rd}}$ thread};
    
    \node (z02) [below=of z01] {};
    \node (z021) [draw,right=of z02, minimum width=\breite cm] 	{{$\mathbf{f}(\mathbf{U}_{n})$}};

    \node (z03) [below=of z02] {};
    \node (z031) [draw,right=of z03, minimum width=\breite cm]	{{$ w^\ast=\mathbf{U}_{n} + k\mathbf{f}(\mathbf{U}_{n}) $}};

    \node (z04) [below=of z03] {};
    \node (z041) [draw,right=of z04, minimum width=\schmal cm] 	{$ (\mathbf{I}_m+ \frac{k}{3} \mathbf{A}_1) a_{1} = \mathbf{U}_{n} $};
    \node (z042) [draw,right=of z041, minimum width=\schmal cm]	{$ (\mathbf{I}_m+ \frac{k}{4} \mathbf{A}_1) b_{1} = \mathbf{U}_{n} $};
    \node (z043) [draw,right=of z042,minimum width=\schmal cm] 	{$ (\mathbf{I}_m+ k \mathbf{A}_1) w^{\ast\ast} = w^\ast $};

    \node (z05) [below=of z04] {};
    \node (z051) [draw,right=of z05, minimum width=\breite cm]{$ c_1 = 9a_1 - 8b_1$};
    
    \node (z06) [below=of z05] {};
    \node (z061) [draw,right=of z06, minimum width=\schmal cm]{$ (\mathbf{I}_m+ \frac{k}{3} \mathbf{A}_1) a_{2} = \mathbf{f}(\mathbf{U}_{n}) $};
    \node (z062) [draw,right=of z061,minimum width=\schmal cm]{$ (\mathbf{I}_m+ \frac{k}{4} \mathbf{A}_1) b_{2} = \mathbf{f}(\mathbf{U}_{n}) $};
    \node (z063) [draw,right=of z062,minimum width=\schmal cm]{$ (\mathbf{I}_m+ k \mathbf{A}_2) \mathbf{U}^\ast_{n+1}= w^{\ast\ast} $};
    
    \node (z07) [below=of z06] {};
    \node (z071) [draw,right=of z07, minimum width=\breite cm]{$ c_2 = 9a_2 - 8b_2 $};
    
    \node (z08) [below=of z07] {};
    \node (z081) [draw,right=of z08, minimum width=\breite cm]{$ f^\ast=\mathbf{f}(\mathbf{U}^\ast_{n+1}) $};
    
    \node (z09) [below=of z08,minimum height=2cm] {};
    \node (z091) [draw,right=of z09,minimum width=\breite cm,minimum height=2cm]{\begin{tabular}{cc}
								$ f_1 = 9 c_1 + 2 k c_2 + k f^\ast $ \\ \\ 
								$ f_2 = -8 c_1 - \frac{3}{2} k c_2 - \frac{k}{2} f^\ast $ \\
								\end{tabular}};
    

    \node (z11) [below=of z09] {};
    \node (z111) [draw,right=of z11, minimum width=\schmal cm]{$ (\mathbf{I}_m+ \frac{k}{3} \mathbf{A}_2) d_1 = f_1 $};
    \node (z112) [draw,right=of z111,minimum width=\schmal cm]{$ (\mathbf{I}_m+ \frac{k}{4} \mathbf{A}_2) d_2 = f_2 $};
    \node (z113) [draw,right=of z112,minimum width=\schmal cm]{};
    \node (z1132)[draw,cross out,right=of z112,minimum width=\schmal cm]{};

    \node (z12) [below=of z11] {};
    \node (z121) [draw,right=of z12, minimum width=\breite cm]{$ \mathbf{U}_{n+1} = d_1 + d_2 $};
    
    \node (z13) [below=of z12] {};
    
    \path [draw,->] (z01) -- (z13);
    \node (pfeil1) [left=of z07]{};
    \node (pfeil2) [left=of pfeil1,rotate=90]{Time};
  \end{tikzpicture}
  }
  \caption{\label{flow2d}Flow chart for the implementation of the ETD-RDP-IF scheme in parallel using three threads for the two dimensional case.}
\end{figure}

 The most time consuming part is the numerical solution of the large linear systems of the form $(\mathbf{I}_m+\frac{k}{\gamma}\mathbf{A}_i)x=b$ with $\gamma\in \{1,3,4\}$ and $\mathbf{A}_i$ given by (\ref{3dcase}) for the 
 three-dimensional case and (\ref{Ai}) for the general case. We focus our explanations directly on the matrices $\mathbf{A}_i\in \mathbb{R}^{m\times m}$ with $m=s\cdotp p^d$.
 These are sparse having only a few diagonals occupied with non-zero elements. Hence, only those diagonals have to be stored. 
 In fact, it is easy to see that the matrix $\mathbf{A}_1$ has three diagonals for both the Dirichlet and Neumann boundary condition case. 
 Precisely, it has a main diagonal (the location is zero) as well as an upper and lower diagonal located at $\pm s$ elements apart from the main diagonal, written compactly here as 
 $\mathbf{A}_1=\mathrm{sparse}(\mathbf{T},[-s,0,s])$ where the matrix $\mathbf{T}$ contains the three vectors $\ell^{(1)}=(0,\ldots,0,\ell_{s+1},\ldots,\ell_m)^\top$, 
 $d^{(1)}=(d_{1},\ldots,d_s)^\top$, and $u^{(1)}=(u_1,\ldots, u_{m-s},0,\ldots,0)^\top$ all of size $s\cdotp p^3$ padded with zeros accordingly.
 Similarly, the matrices $\mathbf{A}_2$ and $\mathbf{A}_3$ for the Dirichlet and Neumann boundary case are sparse and of the form $\mathbf{A}_2=\mathrm{sparse}(\mathbf{T},[-s\cdotp p,0,s\cdotp p])$ and 
 $\mathbf{A}_3=\mathrm{sparse}(\mathbf{T},[-s\cdotp p^2,0,s\cdotp p^2])$, respectively. For the periodic boundary condition case, we have sparse matrices with five diagonals of the form 
 $\mathbf{A}_1=\mathrm{sparse}(\mathbf{T},[-s\cdotp(p-1),-s,0,s,s\cdotp(p-1)])$ with $\mathbf{T}=[\ell^{(2)},\ell^{(1)},d^{(1)},u^{(1)},u^{(2)}]$. Similarly, we have the matrix
 $\mathbf{A}_2=\mathrm{sparse}(\mathbf{T},[-s\cdotp p(p-1),-s\cdotp p,0,s\cdotp p,s\cdotp p(p-1)])$, 
 and finally 
 $\mathbf{A}_3=\mathrm{sparse}(\mathbf{T},[-s\cdotp p^2(p-1),-s\cdotp p^2,0,s\cdotp p^2,s\cdotp p^2(p-1)])$.
 
 The band structure for the $d$-dimensional case is now obvious. For  $i\in\{1,\ldots,d\}$, we have $\mathbf{A}_i=\mathrm{sparse}(\mathbf{T},[-s\cdotp p^{i-1},0,s\cdotp p^{i-1}])$
 for the Dirichlet and Neumann boundary case and $\mathbf{A}_i=\mathrm{sparse}(\mathbf{T},[-s\cdotp p^{i-1}(p-1),-s\cdotp p^{i-1},0,s\cdotp p^{i-1},s\cdotp p^{i-1}(p-1)])$ for the periodic boundary case.
 
 A straight-forward $\mathbf{LU}$-decomposition of the matrices $\mathbf{I}_m+\frac{k}{\gamma}\mathbf{A}_1$, $\mathbf{I}_m+\frac{k}{\gamma}\mathbf{A}_2$, and $\mathbf{I}_m+\frac{k}{\gamma}\mathbf{A}_3$ without 
 pivoting gives lower and upper matrices with the same structure without any fill-ins,
 since 
 the matrices are all diagonal dominant if $k$ is small enough. 
 Hence, we can adapt the well-known Thomas algorithm to solve tridiagonal systems of the form $\mathrm{sparse}(\mathbf{T},[-1,0,1])$ 
 (see for example \cite{conte}) to derive the following algorithm in order to solve systems of the form $\mathrm{sparse}(\mathbf{T},[-w,0,w])$. 
 
 Precisely, we have to compute for the sparse matrices from above with given diagonals $\ell^{(1)}$, $d^{(1)}$, and $u^{(1)}$ with offset
 $w$ and given right-hand side $b$ the following steps. For $i=1,\ldots,w$, we compute  $\alpha_i=u_i^{(1)}/d_i^{(1)}$ and $\beta_i=b_i/d_i^{(1)}$ as well as for $i=w+1,\ldots,m-w$ we have 
 $\alpha_i=u_i^{(1)}/\{d_i^{(1)}-\alpha_{i-w}\cdotp l_i^{(1)}\}$ and for $i=w+1,\ldots,m$ we have $\beta_i=\{b_i-\beta_{i-w}\cdotp l_i^{(1)}\}/\{d_i^{(1)}-\alpha_{i-w}\cdotp l_i^{(1)}\}$. Then, we assign 
 $x_{m-i+1}=\beta_{m-i+1}$ for $i=1,\ldots, w$ and compute $x_i=\beta_i-\alpha_i\cdotp x_{i+w}$ for the decreasing $i=m-w,\ldots,1$.

 Now, we focus on the periodic boundary case. As shown above, the matrices have five diagonals. If the diagonals have the same distance from each other, one could derive the factorization explicitly, and
 use forward and backward substitution similar to the variant of the Thomas algorithm explained before. However, this is not the case here. Using an $\mathbf{LU}$-decomposition shows that a few fill-ins are generated and therefore a variant of the Thomas 
 algorithm is not applicable. One could use the well-known Sherman-Morrison-Woodbury formula (see \cite{golub}), since the periodic case in 1D with $s=1$ is a rank-2 update of the Dirichlet case. Hence, in this special situation, 
 it is possible to obtain an algorithm to compute explicitely $(\mathbf{I}_m+\frac{k}{\gamma}\mathbf{A}_1)$ using the variant of the Thomas algorithm as explained before. 
 But the extension to 2D already gives a rank-2$p$ update (that means one has to solve a linear systems with $2p$ right hand sides using the variant of the Thomas algorithm). The situation gets even more complicated considering 
 $(\mathbf{I}_m+\frac{k}{\gamma}\mathbf{A}_2)$ in 2D or in higher dimensions with $s>1$.
 
 If we consider $s=1$ and $d=2$, then it is easy to see that $(\mathbf{I}_m+\frac{k}{\gamma}\mathbf{A}_1)$ is a diagonal block matrix, where each block is a cyclic Toeplitz matrix. 
 Likewise, $(\mathbf{I}_m+\frac{k}{\gamma}\mathbf{A}_2)$ is a cyclic Toeplitz matrix. That means, we can apply the Fourier transformation to efficiently solve the linear system. Precisely, if $a$ is the first colum of a cyclic Toeplitz matrix, 
 then it holds $(\mathcal{F}_m a) \cdotp (\mathcal{F}_m x)=\mathcal{F}_m b$ and hence $x$ is given by $x=\mathcal{F}_m^{-1}\left((\mathcal{F}_m b)./(\mathcal{F}_m a)\right)$ (see for example \cite{ChenSiam}). Here, $./$ means component-wise division.
 This is faster in computation than using the variant of the Thomas algorithm after applying the Sherman-Morrison-Woodbury formula. Likewise, we have for $s=1$ and $d=3$ that $(\mathbf{I}_m+\frac{k}{\gamma}\mathbf{A}_1)$ and 
 $(\mathbf{I}_m+\frac{k}{\gamma}\mathbf{A}_2)$ are cyclic block Toeplitz matrices and $(\mathbf{I}_m+\frac{k}{\gamma}\mathbf{A}_3)$ is a cyclic Toeplitz matrix.

 \begin{remark}
  Again, we would like to stress the fact that we never store the complete matrices in memory, but instead only the elements from the three or five diagonals depending on the problem at hand. 
  The solution of the linear systems for the Dirichlet and Neumann boundary case can be directly obtained through the use of a generalization of the Thomas algorithm taking a serious advantage of the dimensional splitting altogether. 
  Hence, we can avoid to solve the sparse linear systems iteratively through Jacobi, Gau\ss{}-Seidel or Krylov methods. However, the situation is different for the periodic boundary case in the general setting. 
  For the special case $s=1$, we use the Fourier transform to efficiently solve the sparse linear system instead of applying the Sherman-Morrison-Woodbury formula and solving directly by a variant of the Thomas algorithm. 
 \end{remark}

\section{Numerical results and comparison}\label{section8}
In this section, numerical results are presented for a variety of examples both in two and three dimensions and compared with existing methods.
All numerical results are performed on a PC equipped with four Intel cores (i7-4790 CPU @ 3.60GHz) on a socket each of which can have two threads 
(architecture: x86\_64, CPU operation modes: 32-bit and 64-bit, and byte order: little endian). The machine has $32$GB of memory. We use the gfortran Fortran compiler
gcc version 7.4.0 on SUSE Linux (version 15.1) with the optimization option \texttt{-O3 -march=native} and the OpenMP (version 4.5) option \texttt{-fopen\_mp} and the 
Matlab version R2018a. More information of the mentioned compiler, software, and operation system can be obtained at the following links:
https://www.intel.de 
https://gcc.gnu.org
https://www.suse.com
https://www.openmp.org
https://de.mathworks.com
Additionally, we used the Fortran library FFTPACK5.1 for the fast Fourier transformation available under
https://people.sc.fsu.edu/$\sim$jburkardt/f77\_src/fftpack5.1/fftpack5.1.html

\subsection{Enzyme kinetics of Michaelis-Menten type}
The enzyme kinetics of Michaelis-Menten type reaction-diffusion equation in one dimension has been considered in \cite{cherruault} and since then researchers have considered it in higher 
dimensions as a testing scenario (see for example \cite{bhatt}).
It is of the form
\[\frac{\partial u}{\partial t}=\gamma\Delta u-\frac{u}{1+u}\]
where the positive constant $\gamma$ is given. The domain $\Omega$ is assumed to be a unit square and the time interval is $[0,1]$. The initial condition is prescribed to be constant one and the boundary condition is assumed to be homogeneous Dirichlet. 
Solving this problem numerically is known to be challenging due to the discontinuity of the initial and boundary conditions as it can cause spurious oscillations in the solution (cf. \cite{bhatt}).
Clearly, the problem at hand fits into the format (\ref{model}) with $s=1$, $D=\gamma$, and $f(u)=-u/(1+u)$.

First, we show in Table \ref{enztable} that the new method ETD-RDP-IF is indeed second-order as ETD-RDP without splitting (see \cite{asante} for ETD-RDP). 
We use $D=0.2$, fix $h$ quite small as $0.0125$, and 
vary $k$ from $0.05$ to $0.00625$ to compute the error $\epsilon(k)$ measured in the $L^\infty$ norm and the estimated order of convergence $\text{EOC}=\log(\epsilon(k)/\epsilon(k/2))/\log(2)$ for the final time $T=1$. Additionally, we include the CPU times.
\begin{table}[!ht]
\centering
\begin{tabular}{c|ccc|ccc}         
$k$        &  $\epsilon(k)_\text{ETD-RDP}$       &    EOC   &    CPU Time  &    $\epsilon(k)_\text{ETD-RDP-IF}$      &   EOC  &  CPU time\\
\hline
0.05000 &  $5.3337\times 10^{-3}$  &          &    0.38305   &    $7.1838\times 10^{-3}$ &        &  0.02738\\
0.02500 &  $1.5935\times 10^{-3}$  &    1.74  &    0.76822   &    $1.5829\times 10^{-3}$ &   2.18 &  0.04211\\
0.01250 &  $4.3420\times 10^{-4}$  &    1.88  &    1.51854   &    $3.6926\times 10^{-4}$ &   2.10 &  0.07334\\
0.00625 &  $1.1356\times 10^{-4}$  &    1.93  &    3.02891   &    $8.8718\times 10^{-5}$ &   2.06 &  0.13775\\
\hline
\end{tabular}
\caption{\label{enztable}Estimated order of convergence EOC for both the ETD-RDP and ETD-RDP-IF for the enzyme kinetics example using the parameters $D=0.2$, $h=0.0125$, $T=1$ within the unit square. 
Additionally, the CPU times in seconds of the Matlab program are listed.}
\end{table}
As we see in Table \ref{enztable}, we obtain a second-order convergence rate with comparable errors for both methods. 
Above all, we notice a much better performance of the ETD-RDP-IF compared to ETD-RDP. We obtain a speed-up factor of $20$.

Next, we show in Figure \ref{enzeff} the efficiency of the new method ETD-RDP-IF compared with other second-order methods such as ETD-RDP, IMEX-BDF2, IMEX-TR, and IMEX-Adams2 (all methods are implemented in Matlab and run serial). 
\begin{figure}[!ht]
\centering
\includegraphics[width=7.5cm]{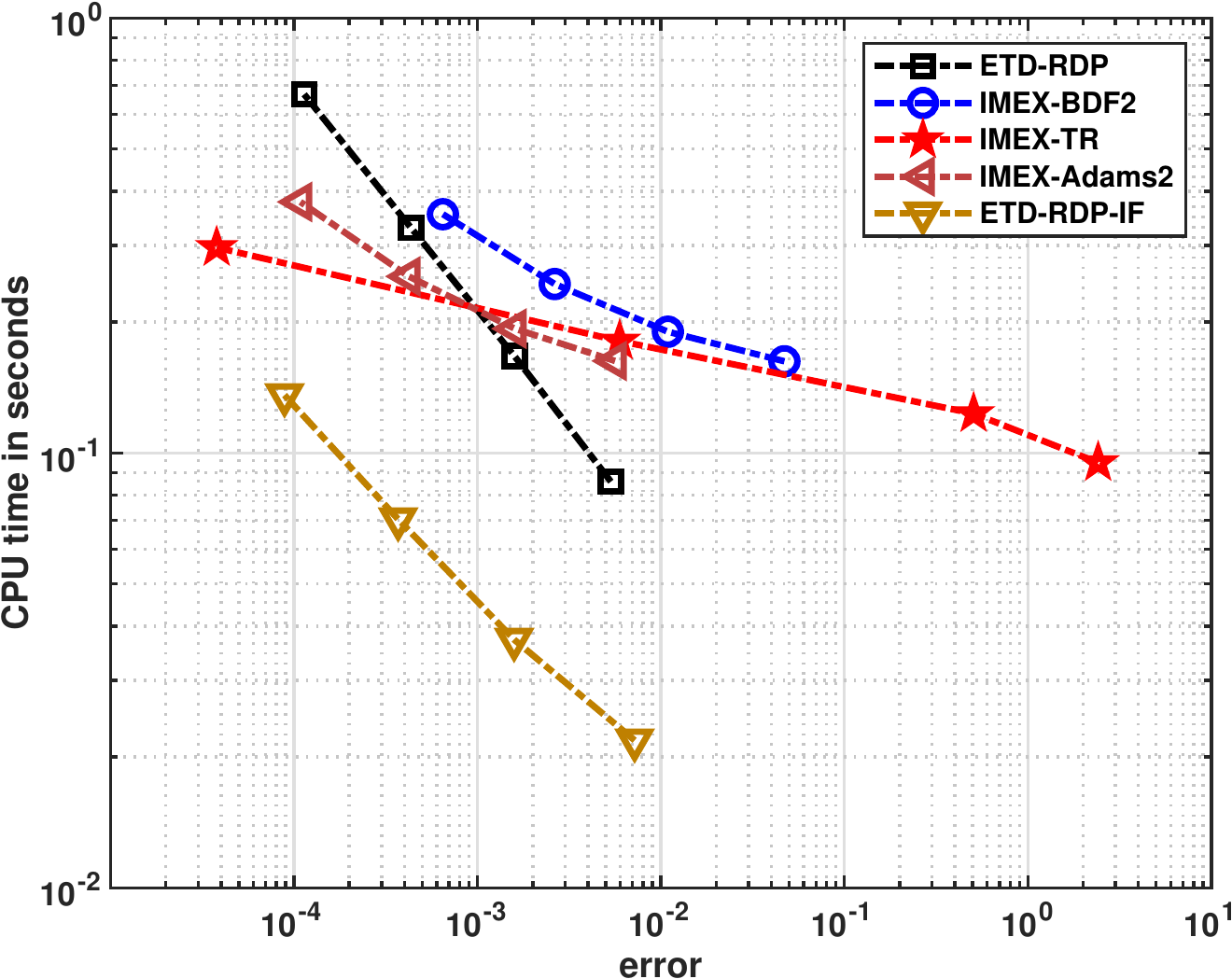}
 \caption{\label{enzeff}Log-Log efficiency plot comparing ETD-RDP and ETD-RDP-IF with different second-order IMEX schemes implemented in Matlab running serial for the enzyme kinetics example using the parameters 
 $D=0.2$ and $T=1$ within the unit square.}
\end{figure}

As we see in Figure \ref{enzeff}, the ETD-RDP method has a comparable efficiency as the different second-order IMEX schemes for this problem. However, the application of dimensional splitting outperforms the IMEX schemes. 
Hence, the ETD-RDP-IF is the most efficient method for this enzyme kinetics example.

Note also that we tried other splitting techniques for the enzyme kinetics example, such as Strang simple splitting and Strang symmetric splitting (see \cite{asantethesis}). It turns out that the 
integrating factor outperforms the other two splitting techniques (as this was also the case in \cite{asante2}).

Now, we show in Table \ref{tableenzyme} the performance of the ETD-RDP-IF method implemented in Matlab in comparison with the Fortran implementation without and with parallelization for a variety of parameter choices for $h$ and $k$ 
using the parameters $D=0.2$ and $T=1$.
\begin{table}[!ht]
\centering
\begin{tabular}{c|rrr}
 $k=h$ & Matlab & Fortran (1 thread) & Fortran (3 threads)\\
 \hline
 $1/100$ &  0.16 & 0.14 & 0.06\\
 $1/200$ &  1.20 & 1.19 & 0.48\\
 $1/400$ &  9.76 & 9.87 & 4.54\\
 $1/800$ & 83.26 & 73.80 & 35.81 \\
 \hline
 \end{tabular}
 \caption{\label{tableenzyme}CPU times in seconds for the ETD-RDP-IF method implemented in Matlab and Fortran (serial and parallelized version) for the enzyme kinetics example using the parameter 
 $D=0.2$ and $T=1$ within the unit square.}
\end{table}
Although the algorithm ETD-RDP-IF is already very efficient in serial compared to other methods, it can be greatly improved first through conversion to Fortran and 
then through parallelization. We obtain a speed-up factor of two. Note that we list the timings of the complete program, i.e. the preparation of matrices (allocation, initialization, etc.) and the timing of 
the ETD-RDP-IF algorithm. 

Note that one could further increase the performance of the Fortran program by using floating point arithmetic with single precision instead of double precision as used above. 
Using three threads and single precision arithmetic within the Fortran program, we obtain the timings
$0.04$, $0.33$, $2.76$, and $23.65$ seconds, respectively. This gives a total speed-up factor of four compared to the Matlab implementation.

\subsection{The Brusselator system}
In this section, we consider the two- and three-dimensional generalization of the one-dimensional reaction-diffusion Brusselator system. 
For more explanations regarding the well-studied model for a hypothetical tri-molecular reaction, we refer the reader to \cite{zegeling}*{p. 526} and the references cited therein for the details 
and the background of this model.
The system in two dimensions reads
\begin{eqnarray*}
 \frac{\partial u_1}{\partial t}&=&D\Delta u_1 +u_1^2u_2-(A+1)u_1+B\\
 \frac{\partial u_2}{\partial t}&=&D\Delta u_2 -u_1^2u_2+Au_1\\
\end{eqnarray*}
where the constants $D$, $A$, and $B$ are given. In our experiment, we consider the domain $\Omega=[0,1]^2$ and $t\in (0,2)$ and the parameters $D=2\times 10^{-3}$, $A=3.4$, and $B=1$ as done in \cite{asante}. 
We prescribe homogeneous Neumann boundary
conditions for both $u_1$ and $u_2$. The initial condition is given by 
$$u_1(x,y,0)=1/2+y\,\qquad u_2(x,y,0)=1+5x\,.$$
Again, we can see in Table \ref{brutable} that the estimated order of convergence agrees with the theoretical convergence order of two where we fixed $h=0.0125$. 
\begin{table}[!ht]
\centering
\begin{tabular}{c|ccc|ccc}         
$k$        &  $\epsilon(k)_\text{ETD-RDP}$       &    EOC   &    CPU Time  &    $\epsilon(k)_\text{ETD-RDP-IF}$      &   EOC  &  CPU time\\
\hline
0.1000 &  $2.3309\times 10^{-1}$  &          &    0.19429   &    $2.3852\times 10^{-1}$ &        &  0.03740\\
0.0500 &  $5.8468\times 10^{-2}$  &    2.00  &    0.38407   &    $4.8944\times 10^{-2}$ &   2.28 &  0.06368\\
0.0250 &  $1.5750\times 10^{-2}$  &    1.89  &    0.76545   &    $1.2898\times 10^{-2}$ &   1.92 &  0.12128\\
0.0125 &  $4.0884\times 10^{-3}$  &    1.95  &    1.53955   &    $3.3223\times 10^{-3}$ &   1.96 &  0.24152\\
\hline
\end{tabular}
\caption{\label{brutable}Estimated order of convergence EOC for both the ETD-RDP and ETD-RDP-IF for the Brusselator example using the parameters $D=2\times 10^{-3}$, $A=3.4$, $B=1$, $h=0.0125$, and $T=2$ 
within the unit square. Additionally, the CPU times in seconds of the Matlab program are listed.}
\end{table}
In Figure \ref{enzeff2} we also show that the ETD-RDP-IF outperforms ETD-RDP and the second-order IMEX schemes IMEX-BDF2, IMEX-TR, and IMEX-Adams2 (all methods are implemented in Matlab and run serial).
\begin{figure}[!ht]
\centering
\includegraphics[width=7.5cm]{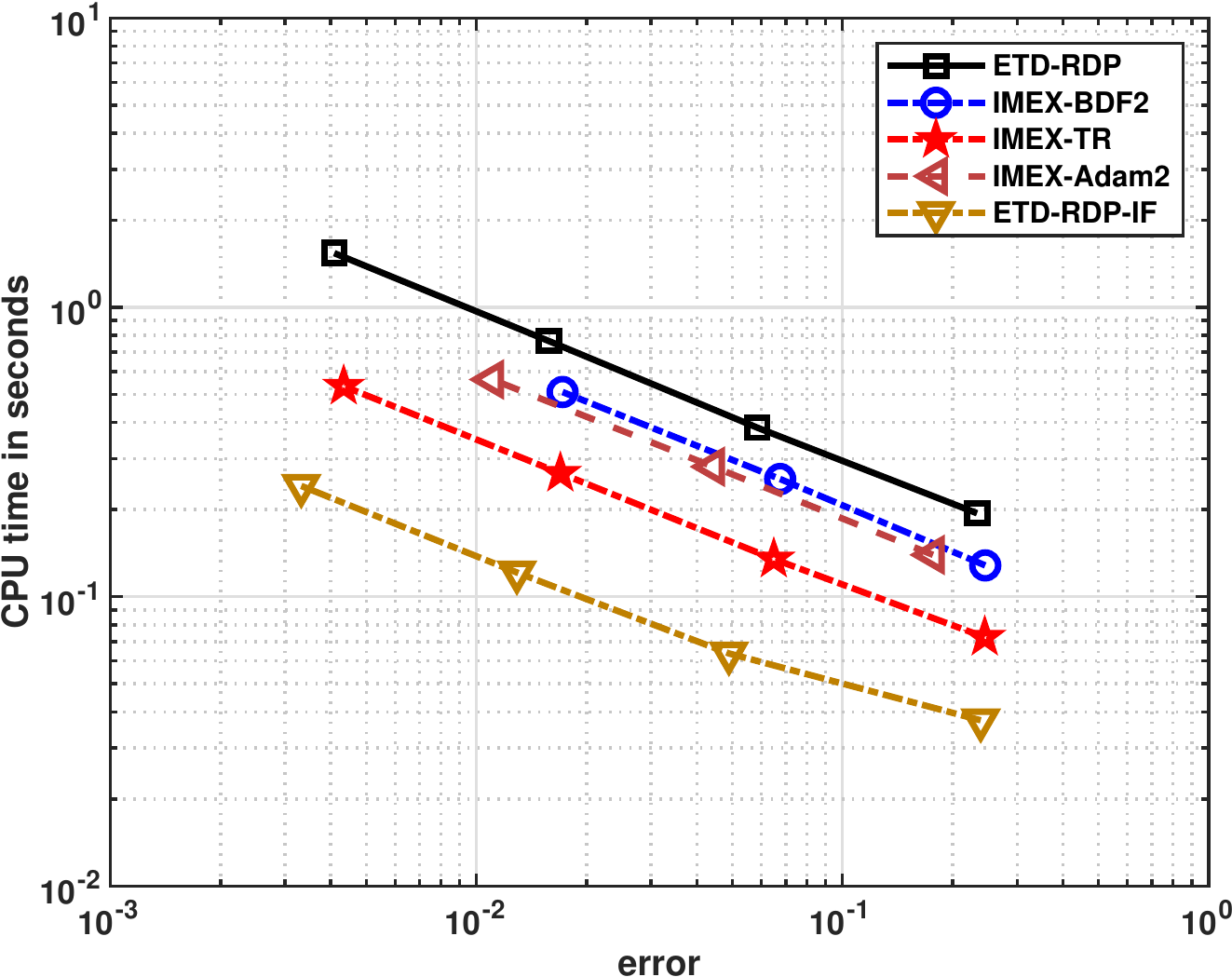}
 \caption{\label{enzeff2}Log-Log efficiency plot comparing ETD-RDP and ETD-RDP-IF with different second-order IMEX schemes implemented in Matlab running serial for the Brusselator example using the parameters 
 $D=2\times 10^{-3}$, $A=3.4$, $B=1$, and $T=2$ within the unit square.}
\end{figure}
The CPU timings can be greatly improved through the Fortran implementation and additionally through parallelization. The CPU timings are listed in Table \ref{tablebrus}.
\begin{table}[!ht]
\centering
\begin{tabular}{c|rrr}
 $k=h$ & Matlab & Fortran (1 thread) & Fortran (3 threads)\\
 \hline
 $1/100$ &  0.51 & 0.36 & 0.16\\
 $1/200$ &  3.87 & 3.32 & 1.60\\
 $1/400$ & 34.05 & 28.99 & 13.92\\
 $1/800$ & 286.06 & 190.75 & 114.19 \\
 \hline
 \end{tabular}
 \caption{\label{tablebrus}CPU times in seconds for the ETD-RDP-IF method implemented in Matlab and Fortran (serial and parallelized version) for the Brusselator example using the parameters $D=2\times 10^{-3}$, $A=3.4$, $B=1$, and $T=2$.}
\end{table}
As we can see, the serial Fortran program is faster than the Matlab program as expected. The parallelized version is almost twice as fast as the serial version.

Finally, we also consider the three-dimensional Brusselator example on the domain $\Omega=[0,1]^3$ with the same initial and boundary conditions and parameters as in test problem 4 within \cite{bhatt}. Precisely, we use homogeneous Neumann boundary conditions and as initial conditions
$u_1(x,y,z,0)=1+\sin(2\pi x)\sin(2\pi y)\sin(2\pi z)$ and $u_2(x,y,z,0)=3$. The parameters are $D=0.02$, $A=1$, $B=2$, and $T=5$. The spatial and time step are given by $h=1/10$ and $k=1/1000$.
In Figure \ref{profiles}, we show a 2D slice through the 3D domain of $u_1$ and $u_2$ for $z=1$ and the final time $T=5$. 
\begin{figure}[!ht]
\centering
\includegraphics[width=6.5cm]{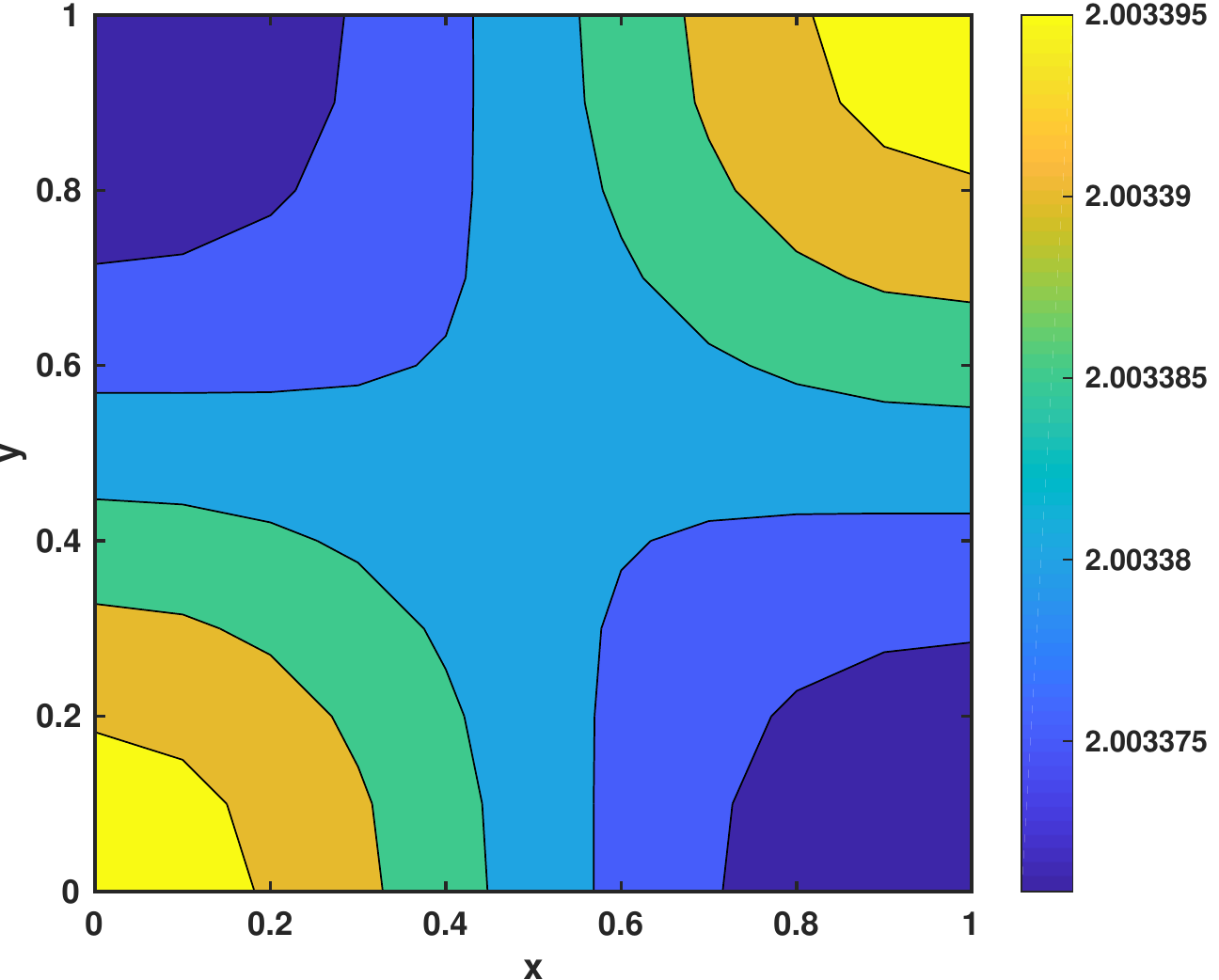}
\includegraphics[width=6.5cm]{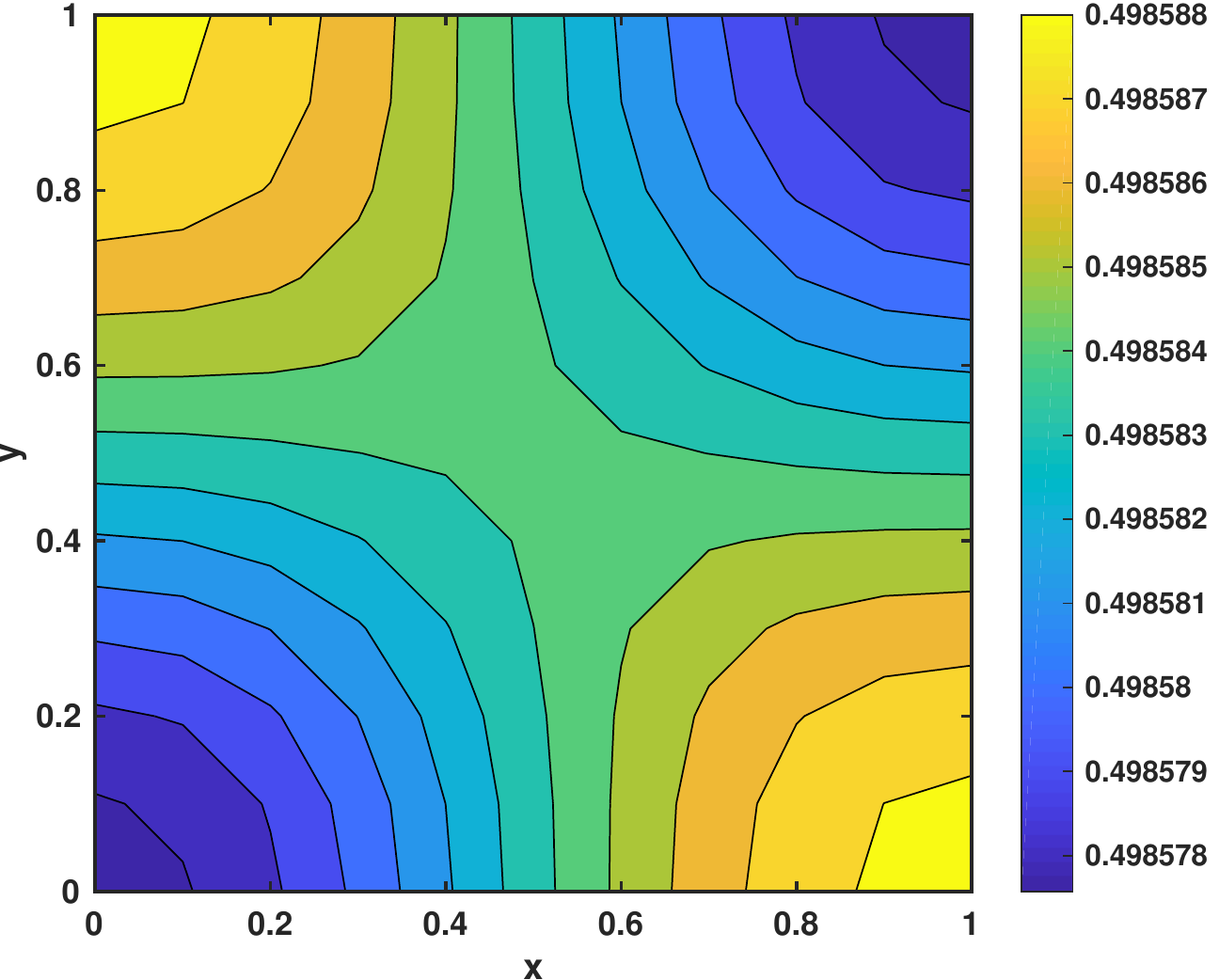}
 \caption{\label{profiles}The concentration profiles of $u_1$ and $u_2$ for $z=1$ at $T=5$ using the parameters $D=0.02$, $A=1$, $B=2$, $h=1/10$ and $k=1/1000$.}
\end{figure}
The two plots agree with the theory of the terminal behavior of the equation proposed in the paper \cite{twizell} since $1-A+B^2 \geq 0$ holds. 
Note that the Matlab program needs $2.87$ seconds to compute the result.
The serial and parallelized Fortran version only need $1.27$ and $0.65$ seconds, respectively.
In Figure \ref{profilespoint} (left) we also show the solutions $u_1$ and $u_2$ at the point $(0.3,0.3,0.3)$ for the time interval $[0,5]$ using the same parameters as before. 
Additionally, we also show in Figure \ref{profilespoint} (right) the solutions $u_1$ and $u_2$ at the point $(1/3,1/3,1/3)$ for the time interval $[0,40]$ using the parameters $D=0.02$, $A=3$, $B=1$, $h=1/10$ and $k=1/1000$. 
Hence, the equation $1-A+B^2 \geq 0$ is violated.
\begin{figure}[!ht]
\centering
\includegraphics[width=6.5cm]{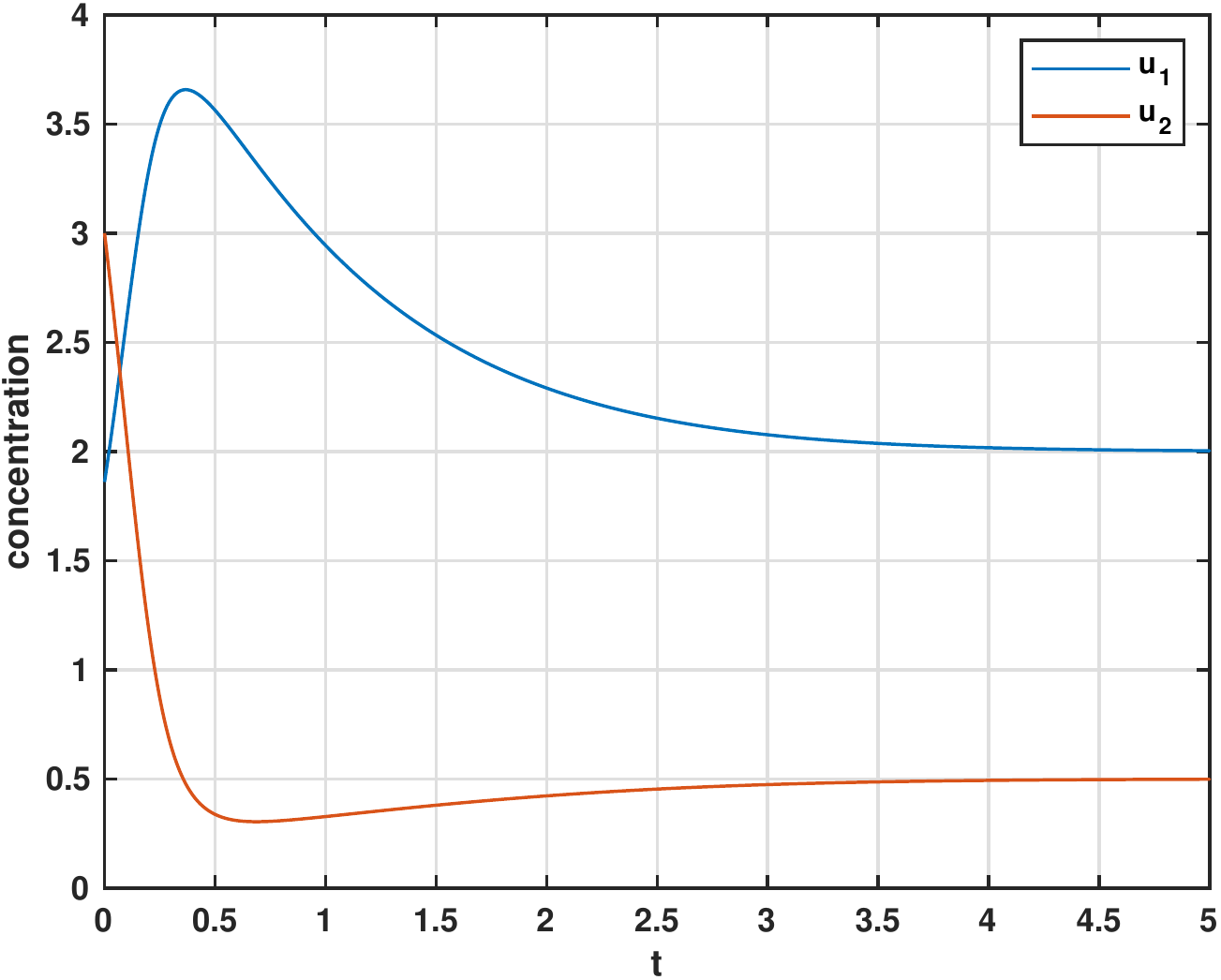}
\includegraphics[width=6.5cm]{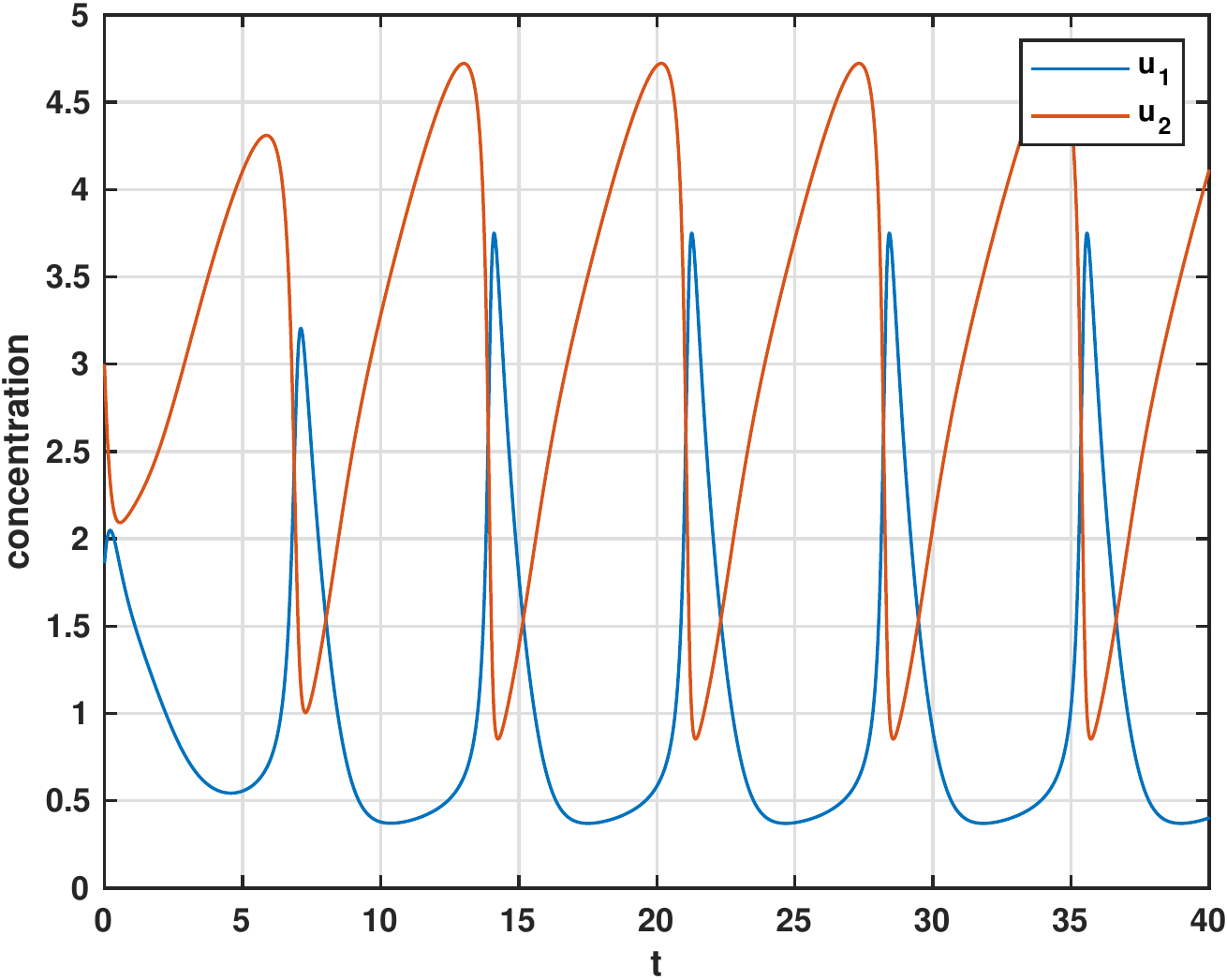}
 \caption{\label{profilespoint}Left: profile of $u_1$ and $u_2$ at the point $(1/3,1/3,1/3)$ for the time interval $[0,5]$ using the parameters $D=0.02$, $A=1$, $B=2$, $h=1/10$ and $k=1/1000$. 
 Right: profile of $u_1$ and $u_2$ at the point $(1/3,1/3,1/3)$ for the time interval $[0,40]$ using the parameters $D=0.02$, $A=3$, $B=1$, $h=1/10$ and $k=1/1000$.}
\end{figure}
As expected, the values approach $B$ and $A/B$ (here $2$ and $1/2$) as $t$ increases, since $1-A+B^2 \geq 0$ is satisfied. However, if we choose $A=3$ and $B=1$, then $1-A+B^2 \geq 0$ is violated 
and we obtain an oscillatory solution. This again is in agreement with \cite{twizell} that the solution does not converge to a fixed concentration. The plots shown in Figure \ref{profilespoint} 
are also in agreement with Figures 10 and 11 within \cite{bhatt}.

\subsection{The complex Ginzburg-Landau equation on a periodic domain}
The complex Ginzburg-Landau equation \cite{ginzburg} has been extensively studied in the physics community. It describes a variety of phenomena such as non-linear 
waves to second-order phase transitions, from superconductivity, superfluidity, and Bose-Einstein condensation to liquid crystals and strings in field theory 
(see \cite{aranson} for a general overview). The complex Ginzburg-Landau equation is given by
\[\frac{\partial u}{\partial t}=u+(1+\mathrm{i}\alpha) \Delta u-\left(1+\mathrm{i}\beta\right)u\left|u\right|^2 \]
where the constants $\alpha$ and $\beta$ are given real-valued parameters describing linear and non-linear dispersion, respectively. Here, the prescribed initial condition is given by a normal random field with 
mean zero and standard deviation one, but could also be a smooth function such as a series of Gaussian pulses. The boundary condition is assumed to be periodic.
Hence, it fits into the format (\ref{model}) with $s=1$, $D=1+\mathrm{i}\alpha$, and $f(u)=u-\left(1+\mathrm{i}\beta\right)u\left|u\right|^2$.

For our experiment, we consider the domain $\Omega=[0,200]^2$ and $t\in (0,100)$ and the parameters $\alpha=0$ and $\beta=1.3$. For the discretization in space, we use $p=400$ ($h=1/2$). For the time step, we use 
$k=1/20$.
To produce the result that is shown in Figure \ref{ginzburg}, we need about 96 seconds with our implementation in Matlab without any parallelization. 
\begin{figure}[!ht] 
\begin{center}
\includegraphics[width=6.5cm]{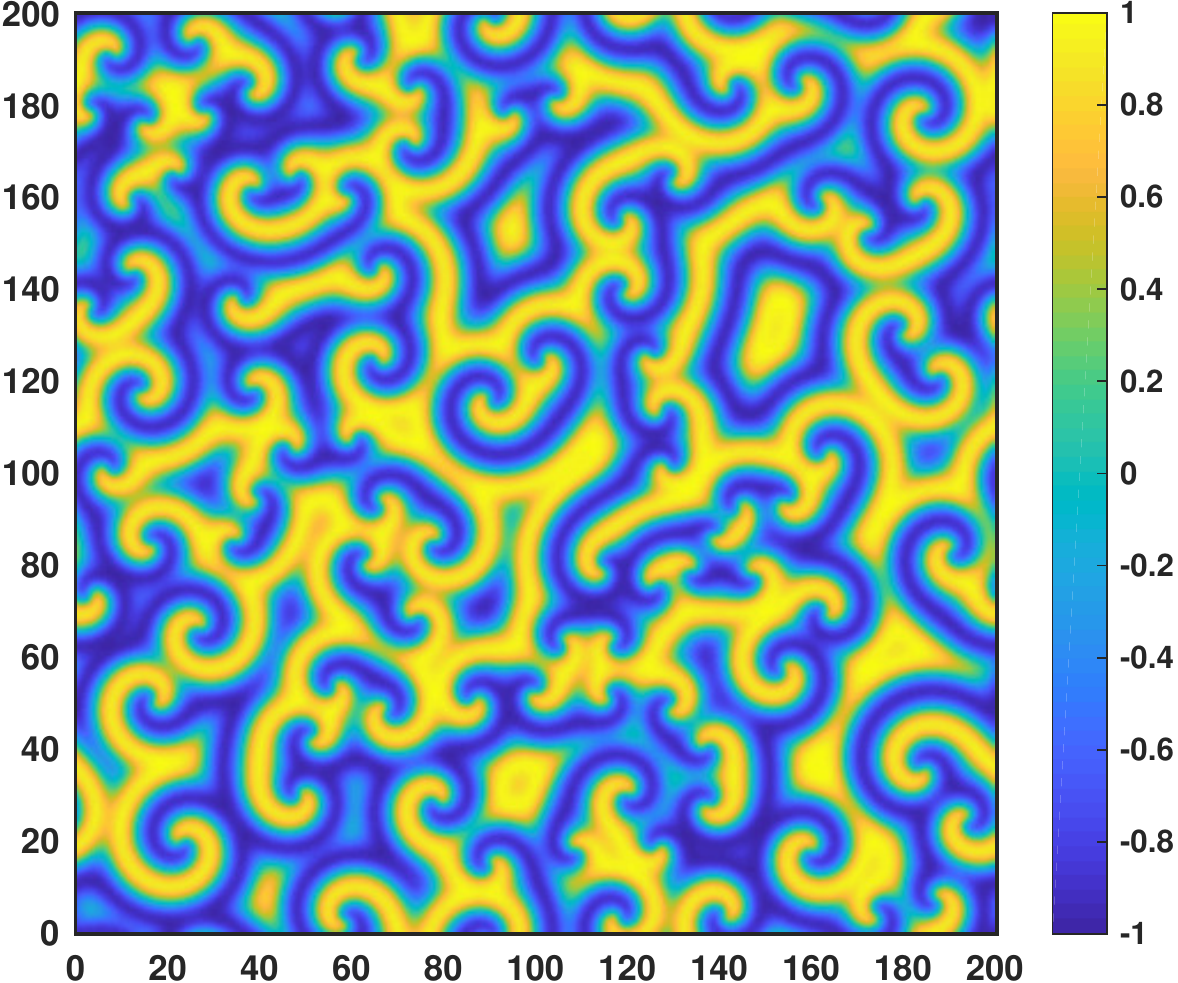}\hspace{0.5cm}
\includegraphics[width=6.5cm]{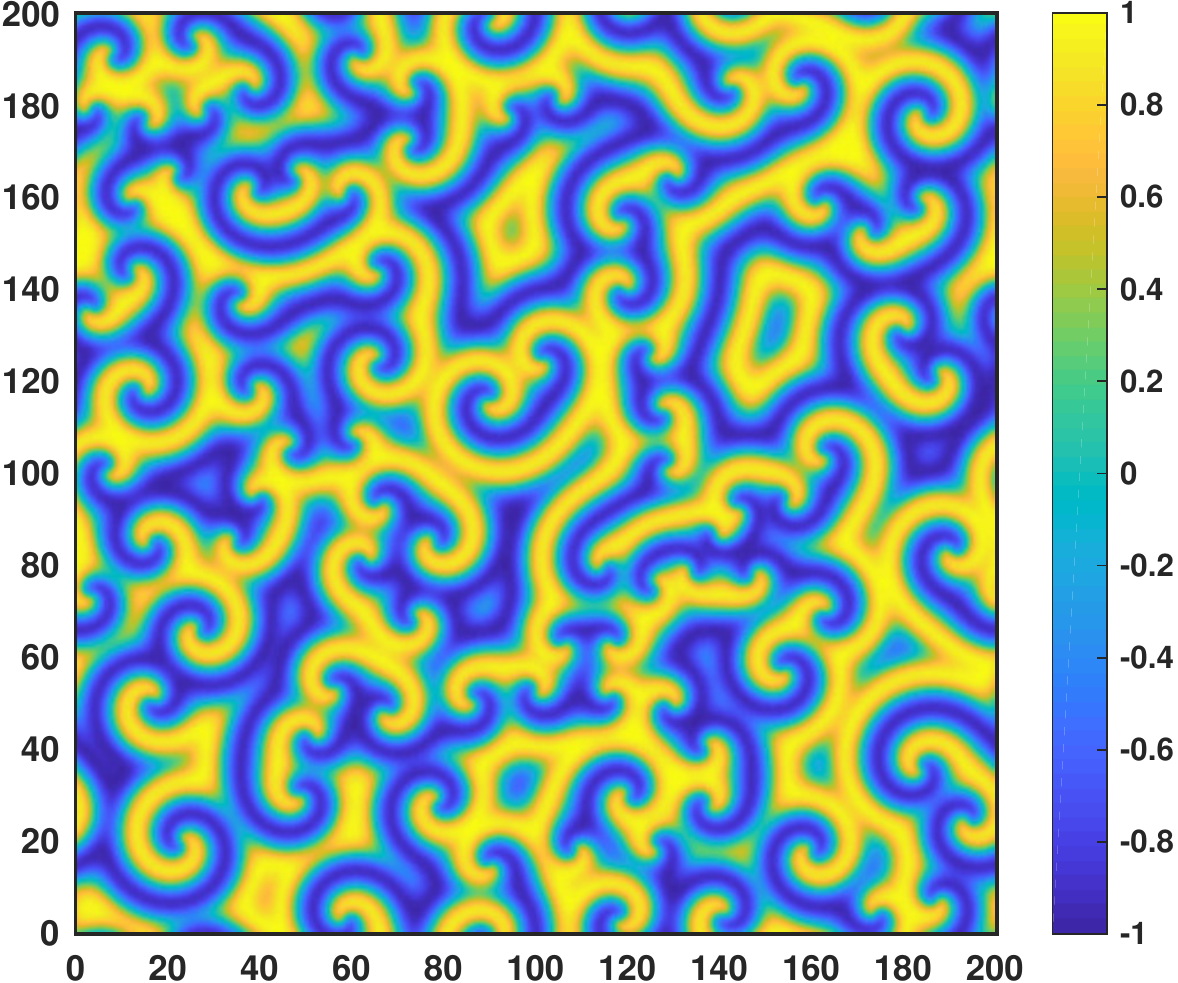}
\end{center}
\caption{\label{ginzburg}Real (left image) and imaginary part (right image) of the solution of the complex Ginzburg-Landau equation on the periodic domain $\Omega=[0,200]^2$ for time $T=100$ with the parameters 
$\alpha=0$, $\beta=1.3$, $p=400$, and $k=1/20$ 
using a standard normal random field as initial condition.}
\end{figure}
Thus, we obtain a comparable performance with the numerical method that is based on 
Fourier spatial discretization and a fourth-order exponential time differencing Runge-Kutta (ETDRK4) method that is usually faster than the standard finite difference scheme (see \cite{kassamtrefethen}). 
However, special attention has to be drawn to the time stepping in Fourier space to avoid aliasing effects. The implementation of the Fourier spectral ETDRK4 method in Matlab 
with the same set of parameters as above needs only 28 seconds (see p. 11 for the Matlab implementation within \cite{kassam}). Of course, we have only compared the timing and 
not the approximation quality of the solution as one should. Further, our method is $L$-stable in comparison to the EDTRK4 and one could create a situation where ETD-RDP-IF outperforms ETDRK4 in the sense 
of approximation quality or more precisely in the sense of CPU time versus accuracy.

In sum, we obtain a stable solution although the initial data are non-smooth. We obtain similar solution patterns, when we use smooth initial boundary conditions such as a series of Gaussian pulses in the form
$u(x,y,0)=\mathrm{e}^{-((x-50)^2+(y-50)^2)/1000}-\mathrm{e}^{-((x-100)^2+(y-100)^2)/1000}+\mathrm{e}^{-((x-100)^2+(y-50)^2)/1000}$ with the same set of parameters as before. 
We again need about 96 ($95.65$ to be precise) seconds in Matlab without parallelization to produce the images shown in Figure \ref{ginzburg2}. 
\begin{figure}[!ht] 
\begin{center}
\includegraphics[width=6.5cm]{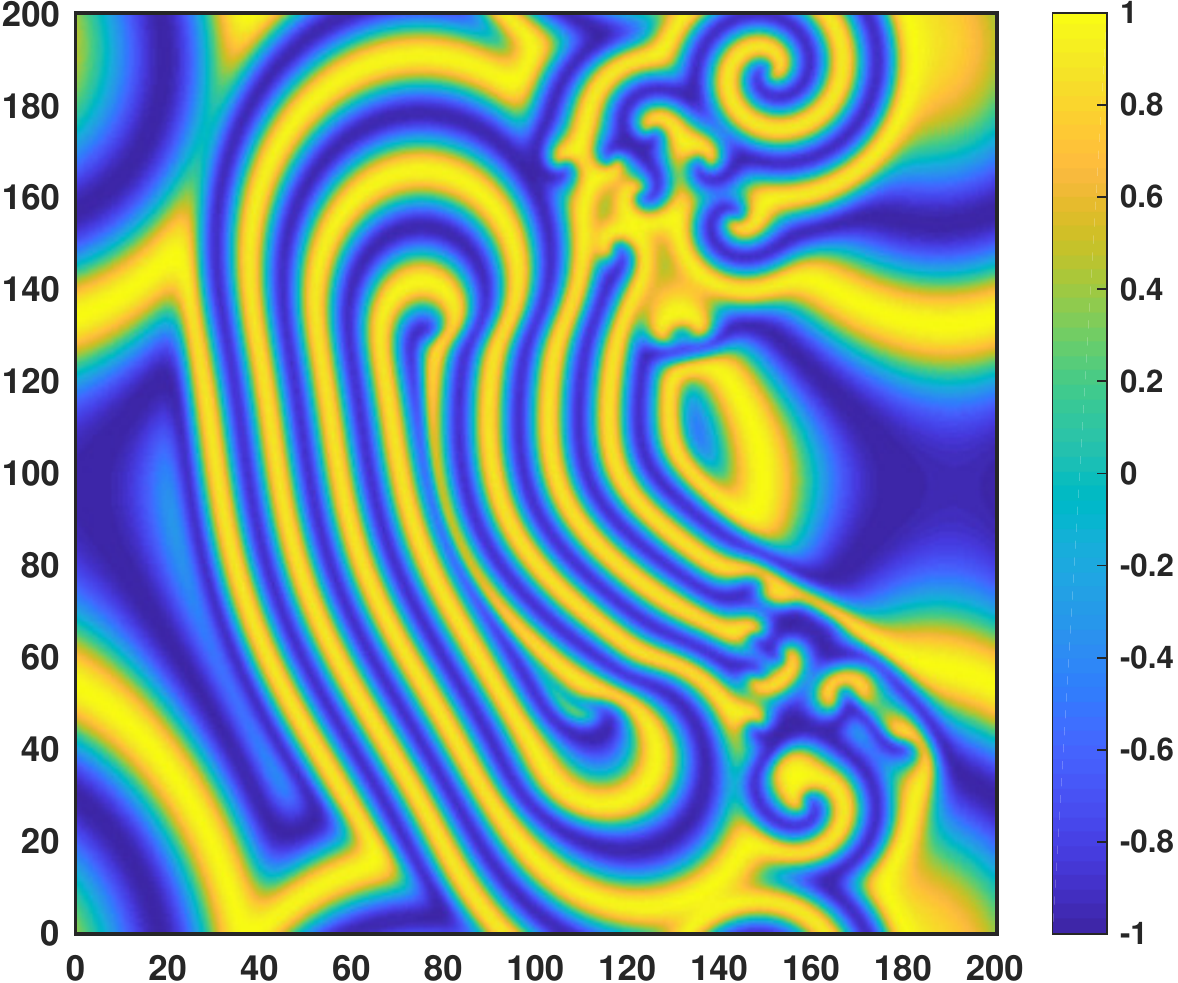}\hspace{0.5cm}
\includegraphics[width=6.5cm]{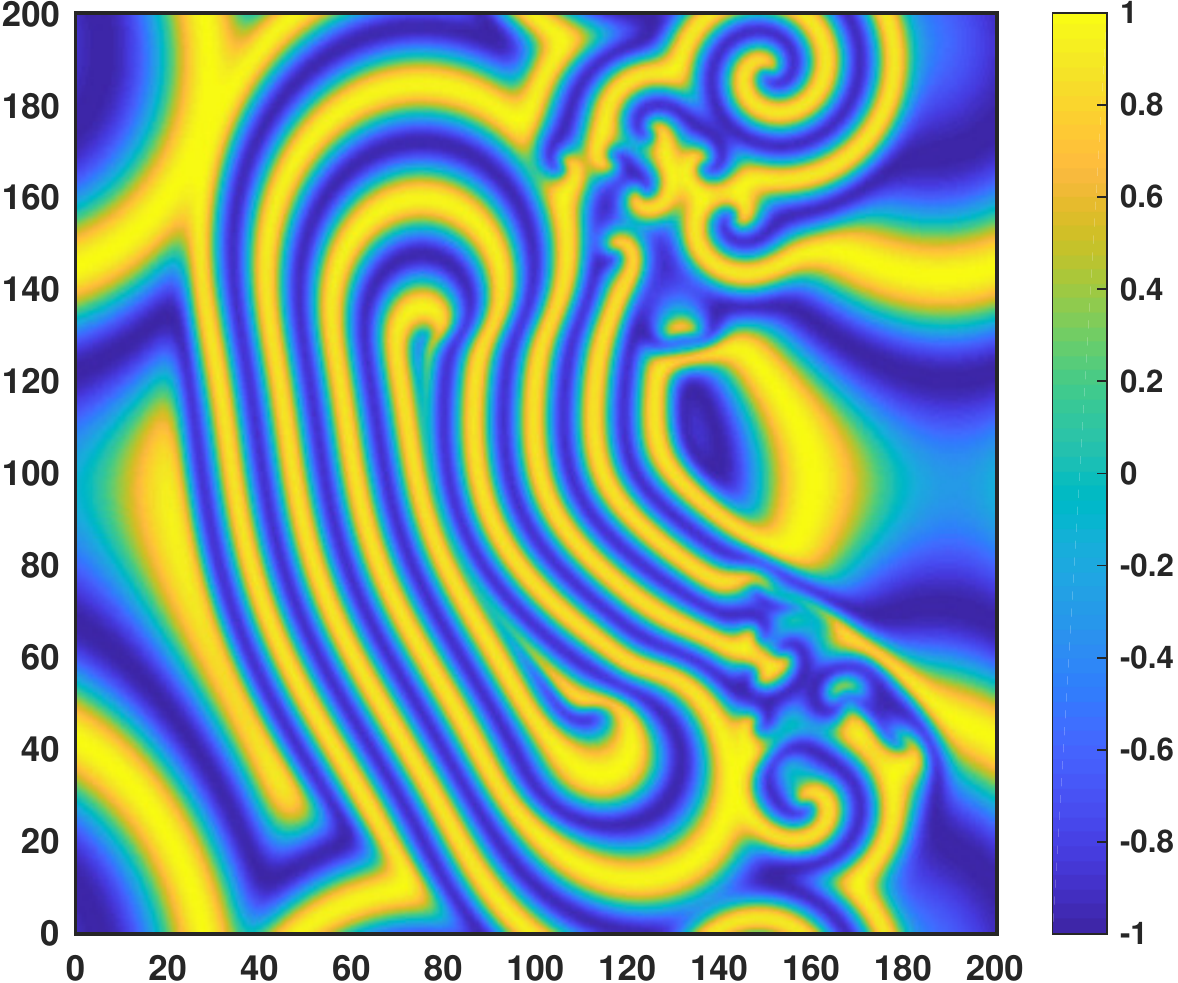}
\end{center}
\caption{\label{ginzburg2}Real (left image) and imaginary part (right image) of the solution of the complex Ginzburg-Landau equation on the periodic domain $\Omega=[0,200]^2$ for time $T=100$ with the parameters 
$\alpha=0$, $\beta=1.3$, $p=400$, and $k=1/20$ 
using a series of Gaussian pulses field as initial condition.}
\end{figure}
The serial Fortran version needs $4.99$ seconds and the parallelized Fortran version needs only $2.76$ seconds. Refer also to the last row in Table \ref{tableginz}.
\begin{table}[!ht]
\centering
\begin{tabular}{rc|rrr}
 $p$ & $h$ & Matlab & Fortran (1 thread) & Fortran (3 threads)\\
 \hline
 $50$ & $4$   & 1.24 & 1.23 & 0.51\\
 $100$& $2$   & 5.17 & 4.95 & 2.04\\
 $200$& $1$   & 23.11 & 20.85 & 10.34 \\
 $400$& $1/2$ & 95.65 & 86.66 & 53.98 \\
 \hline
 \end{tabular}
 \caption{\label{tableginz}CPU times in seconds for the ETD-RDP-IF method implemented in Matlab and Fortran (serial and parallelized version) for the Ginzburg-Landau example on periodic domain $\Omega=[0,200]^2$ 
 using the parameters $\alpha=0$, $\beta=1.3$, $k=1/20$, and $T=100$.}
\end{table}

A similar situation arises when we consider the three-dimensional case. The ETDRK4 using a spectral method needs only 27 seconds in Matlab with the parameters 
$\alpha=0$, $\beta=1.3$, $p=50$ ($h=2$), and $k=1/20$ for time $T=100$ and $\Omega=[0,100]^3$ whereas our method in Matlab (without any parallelization) needs about 119 seconds, which is longer as expected 
due to the second-versus-fourth order of the schemes. The serial and parallelized Fortran program need $108.21$ and $57.87$ seconds, respectively.

We now consider $\Omega=[0,200]^3$ with the same parameters as before and increase 
the parameter $p$ to $200$. That means, we have $8,000,000$ spatial discretization points. Gaussian pulses as initial condition of the form
$u(x,y,z,0)=\mathrm{e}^{-((x-50)^2+(y-50)^2+(z-50)^2)/1000}-\mathrm{e}^{-((x-100)^2+(y-100)^2+(z-100)^2)/1000}$ are used, which
yields Figure \ref{ginzburg3}.
\begin{figure}[!ht] 
\begin{center}
\includegraphics[width=6.5cm]{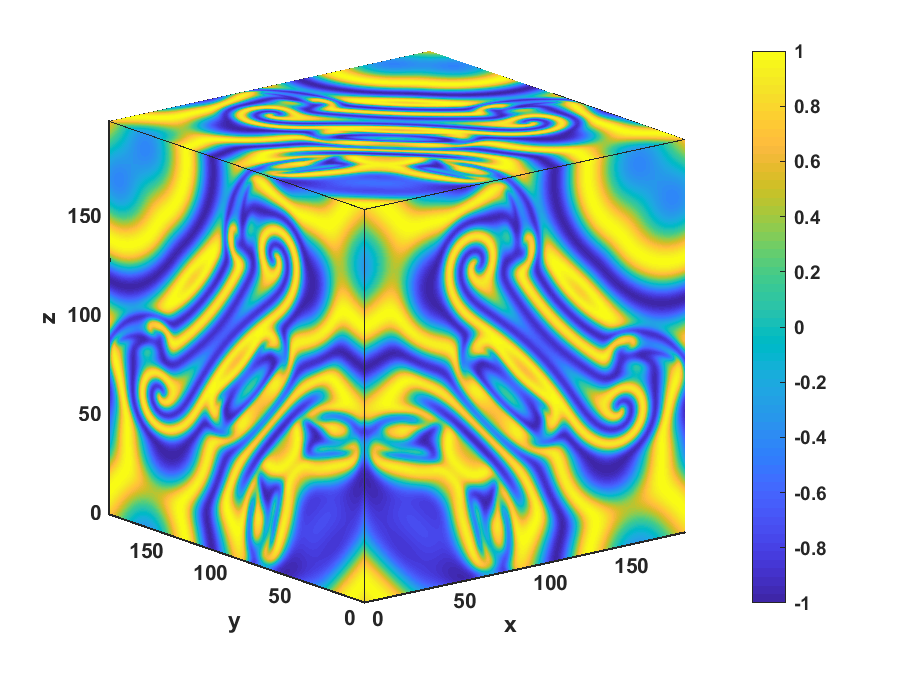}\hspace{0.5cm}
\includegraphics[width=6.5cm]{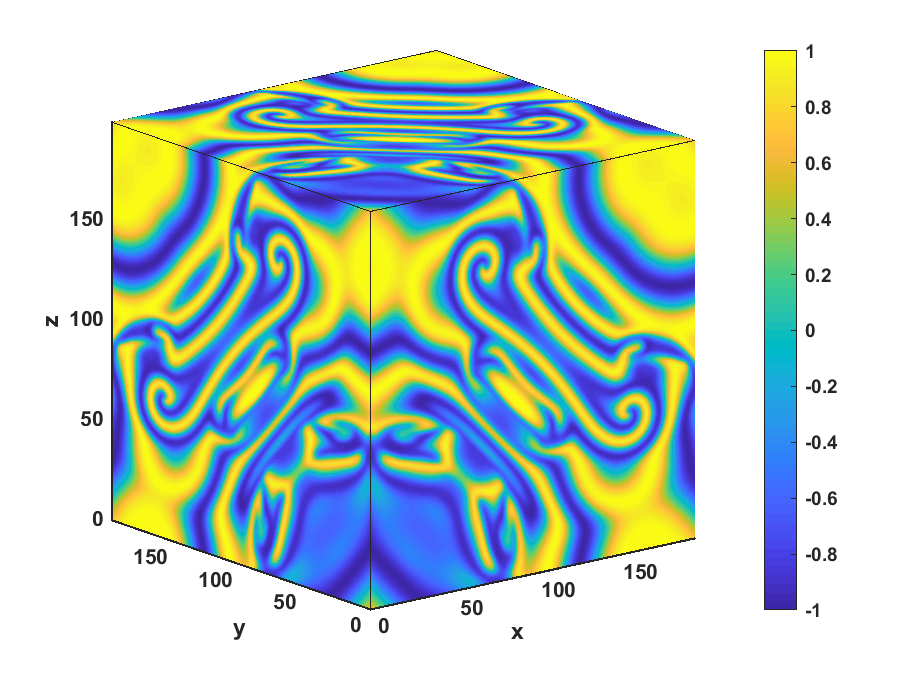}
\end{center}
\caption{\label{ginzburg3}Real (left image) and imaginary part (right image) of the solution of the complex Ginzburg-Landau equation on the periodic domain $\Omega=[0,200]^3$ for time $T=100$ with the parameters 
$\alpha=0$, $\beta=1.3$, $p=200$, and $k=1/20$ 
using a series of Gaussian pulses field as initial condition.}
\end{figure}
The Matlab program needs $9673.13$ seconds whereas the serial and parallelized Fortran program need $8260.91$ and $5287.64$ seconds, respectively. Refer also to the last row in Table \ref{tableginz2}.
\begin{table}[!ht]
\centering
\begin{tabular}{rc|rrr}
 $p$ & $h$ & Matlab & Fortran (1 thread) & Fortran (3 threads)\\
 \hline
 $50$ & $4$  &  119.61&    108.37 &   57.94\\
 $100$& $2$ &  1033.63&    912.41 &  585.91\\
 $200$& $1$ &  9673.13&    8260.91& 5287.64\\
 \hline
 \end{tabular}
 \caption{\label{tableginz2}CPU times in seconds for the ETD-RDP-IF method implemented in Matlab and Fortran (serial and parallelized version) for the Ginzburg-Landau example on periodic domain $\Omega=[0,200]^3$ 
 using the parameters $\alpha=0$, $\beta=1.3$, $k=1/20$, and $T=100$.}
\end{table}

\subsection{Schr\"odinger equation}
The Schr\"odinger equation \cite{schroedinger} is a linear PDE and it describes a state function of a quantum-mechanical system (see for example \cite{griffiths} for more details).
However, in this section, we consider an extension of it. Precisely, we focus on the $d$-dimensional non-linear cubic Schr\"odinger equation of the form
 \begin{equation}\label{Shrodinger}
 \mathrm{i}\Psi_t + \Delta \Psi=q\left(|\Psi|^2\right)\Psi\,, \qquad \mathbf{x}\in [0,1]^d\,,\quad t\in (0,T) 
 \end{equation}
 with given initial condition $\Psi_0(\mathbf{x})=\Psi(\mathbf{x},0)$ and homogeneous Neumann boundary condition. Here, $q$ is a given function of $|\Psi|^2$.
 The wave function $\Psi(\mathbf{x},t)$ can be written as $v(\mathbf{x},t)+ \mathrm{i}w(\mathbf{x},t)$ and hence, we obtain the coupled system of equations
 \begin{align*}
 v_t + \Delta w &= \,\,\,\,q(v^2+w^2)w\\
 w_t - \Delta v &= -q(v^2+w^2)v
 \end{align*}
 which can be written as
\begin{equation*}\label{Shrodinger2}
\begin{pmatrix} v_t\\ w_t \end{pmatrix} + \begin{pmatrix} 0 & 1\\ -1 & 0 \end{pmatrix}\begin{pmatrix}\Delta v\\ \Delta w \end{pmatrix} = q(u^2+w^2)\begin{pmatrix} 0 & 1\\ -1 & 0 \end{pmatrix}\begin{pmatrix} v\\ w \end{pmatrix}.
\end{equation*}
By setting $\mathbf{u} = \begin{pmatrix} v\\ w \end{pmatrix} \text{ and } \textbf{D} = \begin{pmatrix} 0 & 1\\ -1 & 0 \end{pmatrix}$ we can rewrite the system as 
\begin{equation*}
\mathbf{u}_t + \mathbf{D}\Delta \mathbf{u}= q(|\mathbf{u}|^2)\mathbf{D}\mathbf{u}
\end{equation*}
which fits into the format (\ref{model}). However, note that $\textbf{D}$ is not a diagonal matrix. Using the approximation of the Laplacian as in Example \ref{ex1} and Example \ref{ex2}, one can 
again verify similarly to Lemma \ref{lemma2} that the dimensional splitting commutes and hence the ETD-RDP-IF method can be used.

First, we present numerical results for the one-dimensional non-linear cubic Schr\"odinger equation of the form (see also \cite{Bratsos2015}*{Section 4})
 \begin{equation}\label{Shrodinger1D}
 \mathrm{i}\Psi_t + \Psi_{xx} +|\Psi|^2\Psi=0\,, \qquad -L_0 < x< L_1\,,\quad  0<t<T 
 \end{equation}
 with initial condition  \[ \Psi(x,0) = \sqrt{2a}\,\mathrm{exp} \left(\mathrm{i}\,\frac{c}{2}x \right) \mathrm{sech}(\sqrt{a}x) \] 
 and homogeneous Neumann boundary condition. The function  
 \[ \Psi(x,t) = \sqrt{2a}\,\mathrm{exp} \left[i \left(\frac{c}{2}x-\left\{\frac{c^2}{4}-a\right\}t\right) \right] \mathrm{sech}(\sqrt{a}(x-ct))\, \]
 satisfies the Schr\"odinger equation and the Neumann boundary condition as $|x|$ approaches infinity.
 The total mass $M(t)$ and energy $E(t)$ of the wave function are given by 
 \begin{align*}
 M(t) &= \int_{\mathbb{R}} |\Psi(x,t)|^2\,\mathrm{d}x\,,\\
 E(t) &= \int_{\mathbb{R}} \left(|\Psi_x(x,t)|^2 - \frac{1}{2}|\Psi(x,t)|^4 \right)\;\mathrm{d}x\,,
 \end{align*}
 respectively. For the purpose of numerically verifying the conservation of these quantities, we approximate the integrals by the trapezoidal rule (as in \cite{Bratsos2015}*{Section 4.4})
  \begin{align*}
 M(t_n) &\approx \frac{h}{2}\left[ |\mathbf{U}_0^n|^2 + 2 \sum_{i=1}^p  |\mathbf{U}_i^n|^2 +  |\mathbf{U}_{p+1}^n|^2 \right] \\
 E(t_n) &\approx  \frac{h}{2} \left[-\frac{1}{2} (|\mathbf{U}_0^n|^4 +  |\mathbf{U}_{p+1}^n|^4)+ 2  \sum_{i=1}^p  \left( \left|\frac{\mathbf{U}_{i+1}^n- \mathbf{U}_{i-1}^n}{2h}\right|^2 - \frac{1}{2} |\mathbf{U}_i^n|^4  \right)\right]\,.
 \end{align*}
 For the numerical results we used the parameters $a = 0.01$, $c = 0.1$, $T=108$, $h = 1/2$, and $\Delta t = 1/100$ within the interval $[L_0,L_1]=[-80,100]$. 
 A plot of the exact solution (defined on $\mathbb{R}$) and the approximated solution (defined on $[-80,100]$) with computing time less than five seconds in Matlab reveals 
 visually no difference and agrees with the one given in \cite{Bratsos2015}*{Figure 4.1}. 
 The exact values for $M(0)$ and $M(108)$ can be obtained by basic integration techniques from calculus. They are both $2/5=0.4$. 
 Numerically, we obtain $0.399\,999\,954\,123\,281$ and $0.399\,999\,954\,128\,036$, respectively. For the energy $E(0)$ and $E(108)$, we obtain the exact value $-1/3000\approx -0.000\,333\,333\,333\,333$. 
 We obtain numerically $-0.000\,336\,760\,546\,294$ and $-0.000\,336\,768\,976\,796$. Hence, the mass and the energy are conserved asymptotically for the approximated solution obtained via ETD-RDP-IF.
 
 Next, we consider the two-dimensional Schr\"odinger equation within the unit square and the function $q(|\mathbf{u}|^2)=B(x,y)+C(x,y)|\mathbf{u}|^2$ 
 with $B(x,y)=(1-2\pi^2)(1-\cos^2(\pi x)\cos^2(\pi y))$ and $C(x,y)=(1-2\pi^2)$. The exact solution is given by $\Psi(x,y,t) = \mathrm{e}^{-\mathrm{i}t}\cos(\pi x)\cos(\pi y)$ (see also \cite{mehdi}*{p. 754}). 
 We use the parameters $T=1$, $k=0.0125$, and $h=1/78$ to create the same plot as in \cite{mehdi}*{Figure 5} within one second. 
 Note that we obtain similar timing improvements for the parallelized Fortran version in comparison to the serial Matlab program as in the Brusselator example.
 For some recent results on more complex domains, we refer the reader to \cite{speck}.

\section{Conclusion} \label{section9} 
In this article, we have combined the second-order exponential time differencing method (ETD) of second-order with approximating the matrix exponential through (non-Pad\'e) rational approximation having real simple and 
distinct poles (RDP) and the integrating factor (IF)
dimensional splitting technique to obtain the second-order $L$-stable ETD-RDP-IF scheme in two and three dimensions. With this scheme, one can solve non-linear 
reaction-diffusion equations with either Dirichlet, Neumann, 
or periodic boundary conditions in either two or three dimensions. The new scheme outperforms other second-order IMEX schemes such as IMEX-BDF2, IMEX-TR, and IMEX-Adams2 as well as ETD-RDP. 
The excellent performance for 2D and above all 3D examples can be further enhanced through the conversion from Matlab to Fortran. Using basic parallelization techniques increases the efficiency further. 
The source code is available under the link
\begin{center}
\texttt{\href{https://github.com/kleefeld80/ETDRDPIF}{https://github.com/kleefeld80/ETDRDPIF}}
\end{center}
In the future, we intend to extend these ideas to develop methods with higher order. Additionally, incorporating an adaptive time selection procedure to these methods might further increase the performance.
As an alternative one might want to consider and apply parallel-in-time procedures (see for example \cite{Gander2015} for an overview).

\section*{Acknowlegement}
The collaboration of the authors arose in C\'adiz, Spain during the 
international Conference on Computational and Mathematical Methods in Science and Engineering (CMMSE'19). We would like to thank the organizer, Jes\'us Vigo-Aguiar, for valuable support and encouragement.
\bibliography{bibo}

\end{document}